\documentclass[10pt]{amsart}

\usepackage{amsthm}
\usepackage{amssymb}
\usepackage{amsmath}
\usepackage{easywncy}

\numberwithin{equation}{section}
\newtheorem{thm}{Theorem}[section]
\newtheorem{prop}[thm]{Proposition}
\newtheorem{lem}[thm]{Lemma}
\newtheorem{cor}[thm]{Corollary}

\theoremstyle{definition}
\newtheorem{definition}[thm]{Definition}

\newtheorem{rem}[thm]{Remark}

\newcommand{\bk}{\boldsymbol{k}}
\newcommand{\bl}{\boldsymbol{l}}

\newcommand{\shp}{\mathcyr{sh}}

\begin{document}

\title[A new deformation of multiple zeta value]
{A new deformation of multiple zeta value}
\author{Yoshihiro Takeyama}
\address{Department of Mathematics,
    Institute of Pure and Applied Sciences,
    University of Tsukuba, Tsukuba, Ibaraki 305-8571, Japan}
\email{takeyama@math.tsukuba.ac.jp}
\thanks{This work was supported by JSPS KAKENHI Grant Number 22K03243.}

\begin{abstract}
    We introduce a new deformation of multiple zeta value (MZV).
    It has one parameter $\omega$ satisfying $0<\omega<2$ and
    recovers MZV in the limit as $\omega \to +0$.
    It is defined in the same algebraic framework as
    a $q$-analogue of multiple zeta value ($q$MZV)
    by using a multiple integral.
    We prove that our deformed multiple zeta value satisfies
    the double shuffle relations which are satisfied by $q$MZVs.
    We also prove the extended double Ohno relations, which are proved
    for ($q$)MZVs by Hirose, Sato and Seki, by using
    a multiple integral whose integrand contains
    the hyperbolic gamma function due to Ruijsenaars.
\end{abstract}
\maketitle

\setcounter{section}{0}
\setcounter{equation}{0}

\section{Introduction}

We call a (non-empty) tuple of positive integers
an \textit{index}.
An index $\bk=(k_{1}, \ldots , k_{r})$ is said to be
\textit{admissible} if $k_{r}\ge 2$.
For an admissible index
$\bk=(k_{1}, \ldots , k_{r})$,
the \textit{multiple zeta value} $\zeta(\bk)$ (MZV for short) is defined by
\begin{align*}
    \zeta(\bk)=\sum_{0<m_{1}<\cdots <m_{r}}\frac{1}{m_{1}^{k_{1}} \cdots m_{r}^{k_{r}}}.
\end{align*}
Recently, a $q$-analogue of MZV ($q$MZV), which is a one-parameter deformation of MZV,
is studied. For example, the so-called Bradley-Zhao model of it is defined by
\begin{align*}
    \zeta_{q}(\bk)=
    \sum_{0<m_{1}<\cdots <m_{r}}
    \frac{q^{(k_{1}-1)m_{1}+\cdots +(k_{r}-1)m_{r}}}{[m_{1}]^{k_{1}} \cdots [m_{r}]^{k_{r}}},
\end{align*}
where $q$ is a deformation parameter satisfying $0<q<1$ and
$[n]=(1-q^{n})/(1-q)$ is the $q$-integer.
In the limit as $q\to 1$,
it recovers the MZV $\zeta(\bk)$.
It is known that $q$MZVs satisfy
many of the same linear relations of MZVs.

In this paper we introduce a new deformation of MZV,
which we call an $\omega$-deformed multiple zeta value,
and show that it satisfies the double shuffle relations
and the extended double Ohno relations
which are satisfied by $q$MZVs.

We recall briefly the definition of $q$MZVs
and the above-mentioned two relations.
Let $\mathfrak{H}=\mathcal{C}\langle a, b \rangle$
be the non-commutative polynomial ring of two variables $a$ and $b$
whose coefficient ring is $\mathcal{C}=\mathbb{Q}[\hbar]$.
Set $\mathcal{A}=\{\hbar b\}\cup \{ba^{k} \mid k\ge 1\}$.
We denote by $\widehat{\mathfrak{H}^{0}}$
the $\mathcal{C}$-submodule of $\mathfrak{H}$ spanned by $1$ and
the words in $\mathcal{A}$ whose rightmost letter is not $\hbar b$.
We endow $\mathbb{R}$ with $\mathcal{C}$-module structure
such that $\hbar$ acts as multiplication by $1-q$.
Then the $q$MZV is defined as an image of
a $\mathcal{C}$-linear map
$Z_{q}: \widehat{\mathfrak{H}^{0}} \to \mathbb{R}$
(see Section \ref{subsec:qMZV} for details).
For example, the Bradley-Zhao model $\zeta_{q}(\bk)$ is equal to
$Z_{q}(e_{k_{1}} \cdots e_{k_{r}})$,
where $e_{k}=b(a+\hbar)a^{k-1}$ for $k\ge 1$.

The double shuffle relations of $q$MZVs are described as follows
(see, e.g., \cite{Zhaobook}).
On the domain $\widehat{\mathfrak{H}^{0}}$ of the map $Z_{q}$,
two associative and commutative binary operations
called the shuffle product $\shp_{\hbar}$ and
the harmonic product $\ast_{\hbar}$ are defined.
Then the map $Z_{q}$ is a homomorphism with respect to
both of the two products, and hence we have
$Z_{q}(w \, \shp_{\hbar} \, w'-w \ast_{\hbar} w')=0$
for any $w, w' \in \widehat{\mathfrak{H}^{0}}$.
This equality gives linear relations among $q$MZVs
whose coefficients belong to $\mathbb{Q}[1-q]$,
which are called the \textit{double shuffle relations}.

In \cite[Corollary 3.22]{Satoh}, Satoh proves that
the shuffle product and the harmonic product are related to each other
via an anti-involution $\sigma$ defined on
the scalar extension $\mathbb{Q}[\hbar, \hbar^{-1}]\otimes_{\mathcal{C}}\mathfrak{H}$.
This result implies that
the double shuffle relations follow from
the shuffle product relations $Z_{q}(w\,\shp_{\hbar}\,w')=Z_{q}(w)Z_{q}(w')$
and the resummation duality $Z_{q}(\sigma(w))=Z_{q}(w)$ proved in \cite[Theorem 4]{qDS}.

The extended double Ohno relations of $q$MZVs are
obtained by Hirose, Sato and Seki \cite{HSS}.
For an admissible index $\bk=(k_{1}, \ldots , k_{r})$
and non-negative integers $m$ and $n$,
we define the \textit{double Ohno sum} $O_{m, n}(\bk)$ by
\begin{align*}
    O_{m, n}(\bk)=\sum_{\substack{
    m_{1}, \ldots , m_{r}, n_{1}, \ldots , n_{r}\ge 0 \\
    m_{1}+\cdots+m_{r}=m                              \\
    n_{1}+\cdots +n_{r}=n
    }}
    \zeta_{q}(k_{1}+m_{1}+n_{1}, \ldots , k_{r}+m_{r}+n_{r}).
\end{align*}
The extended double Ohno relations are
$\mathbb{Z}$-linear relations of the double Ohno sums,
which also give the same relations for MZVs
in the limit as $q\to 1$.

Although we do not write down an explicit form of
the extended double Ohno relations here,
we state two corollaries
(see \cite{HSS} for details).
For this purpose,
we recall the definition of the dual index $\bk^{\dagger}$
of an admissible index $\bk$.
Write an admissible index $\bk$ in the form
\begin{align*}
    \bk=(\{1\}^{b_{1}-1}, a_{1}+1, \ldots , \{1\}^{b_{r}-1}, a_{r}+1)
\end{align*}
with positive integers $a_{1}, \ldots , a_{r}$ and
$b_{1}, \ldots , b_{r}$.
Then its dual $\bk^{\dagger}$ is defined by
\begin{align*}
    \bk^{\dagger}=(\{1\}^{a_{r}-1}, b_{r}+1, \ldots , \{1\}^{a_{1}-1}, b_{1}+1).
\end{align*}
For example, $(1, 3, 1, 1, 2)^{\dagger}=(4, 1, 3)$.
For an admissible index
$\bk=(k_{1}, \ldots , k_{r})$ and non-negative integer $m$,
the \textit{Ohno sum} $O_{m}(\bk)$ is defined by
\begin{align*}
    O_{m}(\bk)=O_{m, 0}(\bk)=
    \sum_{\substack{m_{1}, \ldots , m_{r}\ge 0 \\
    m_{1}+\cdots +m_{r}=m}}
    \zeta_{q}(k_{1}+m_{1}, \ldots , k_{r}+m_{r}).
\end{align*}
The first corollary is the Ohno relations:
\begin{align*}
    O_{m}(\bk)=O_{m}(\bk^{\dagger})
\end{align*}
for any admissible index $\bk$ and
non-negative integer $m$.
Originally they are proved by Ohno \cite{O} for MZVs,
and Bradley proved them for $q$MZVs in \cite{Bradley}.
The second corollary is the double Ohno relations:
\begin{align*}
    O_{m, n}(\bk)=O_{m, n}(\bk^{\dagger})
\end{align*}
for any non-negative integers $m, n$
and any admissible index $\bk$ of the form
\begin{align*}
    \bk=(\{2\}^{n_{0}}, 1, \{2\}^{n_{1}}, 3,
    \ldots ,
    \{2\}^{n_{2d-2}}, 1, \{2\}^{n_{2d-1}}, 3,
    \{2\}^{2d})
\end{align*}
with some non-negative integers
$d, n_{0}, \ldots , n_{2d}$,
which is called
a \textit{BBBL-type} index after \cite[Section 7]{BBBL}.
In \cite{HMOS} it is proved that
the double Ohno relations follow from the Ohno relations.
The extended double Ohno relations,
which contain the Ohno relations and their double version,
are proved by using the technique called
the connected sum method due to
Seki and Yamamoto \cite{SekiYamamoto}.

The $\omega$-deformed multiple zeta value,
which we newly introduce in this paper,
is defined as follows.
A deformation parameter which we denote by $\omega$ is assumed to satisfy $0<\omega<2$.
We endow the complex number field $\mathbb{C}$
with $\mathcal{C}$-module structure such that
$\hbar$ acts as multiplication by $2\pi i \omega$.
Then we will define, by abuse of notation, a $\mathcal{C}$-linear map
$Z_{\omega}: \widehat{\mathfrak{H}^{0}} \to \mathbb{C}$
and call the value $Z_{\omega}(w)$ with $w \in \widehat{\mathfrak{H}^{0}}$
an \textit{$\omega$-deformed multiple zeta value}
($\omega$MZV for short).
For example, the $\omega$MZV corresponding to the Bradley-Zhao model
of $q$MZV is given by
\begin{align*}
    \zeta_{\omega}(\bk)
     & =Z_{\omega}(e_{k_{1}} \cdots e_{k_{r}}) \\
     & =
    (2\pi i \omega)^{|\bk|}
    \prod_{a=1}^{r}\int_{-\epsilon+i\,\mathbb{R}}\frac{dt_{a}}{e^{2\pi i t_{a}}-1}
    \prod_{a=1}^{r}
    \frac{e^{2\pi i \omega (k_{a}-1)(t_{1}+\cdots +t_{a})}}
    {(1-e^{2\pi i \omega (t_{1}+\cdots +t_{a})})^{k_{a}}},
\end{align*}
where $|\bk|=\sum_{a=1}^{r}k_{a}$ is the \textit{weight} of
$\bk=(k_{1}, \ldots, k_{r})$
and $0<\epsilon<\min{(1, 1/r\omega)}$.
The above integral is absolutely convergent if $k_{r}\ge 2$
and does not depend on $\epsilon$.
In the limit as $\omega \to +0$,
it recovers the MZV $\zeta(\bk)$.
Similar deformation is introduced by Kato and the author
for multiple $L$-values except MZVs in \cite{KT}.

The main results of this paper are as follows.
First, $\omega$MZVs satisfy
the shuffle product relations
$Z_{\omega}(w\,\shp_{\hbar}\,w')=Z_{\omega}(w)Z_{\omega}(w')$
and the duality $Z_{\omega}(\sigma(w))=Z_{\omega}(w)$
for any $w, w' \in \widehat{\mathfrak{H}^{0}}$.
Hence $\omega$MZVs satisfy the same double shuffle relations as $q$MZVs,
where $\hbar$ acts as multiplication by $2\pi i \omega$.
Second, the above-mentioned $\omega$MZVs $\zeta_{\omega}(\bk)$
satisfy the extended double Ohno relations
in the same form as ($q$)MZVs.
In the proof we make use of a multiple integral
which corresponds to the connected sum given in \cite{HSS},
whose integrand contains the hyperbolic gamma function
defined by Ruijsenaars \cite{R1}.

The paper is organized as follows.
Section \ref{sec:preliminaries} gives some preliminaries
on $q$MZVs and their relations.
In Section \ref{sec:def-omegaMZV} we define
the $\omega$MZVs, and prove the double shuffle relations
in Section \ref{sec:DS-omegaMZV}.
In Section \ref{sec:extended-DO} we prove
the extended double Ohno relations of $\omega$MZVs.
Appendix \ref{sec:app1} and Appendix \ref{sec:app2} provide
technical proofs of propositions in this paper.


\section{Preliminaries}\label{sec:preliminaries}

\subsection{Shuffle product and Harmonic product}

Let $\hbar$ be a formal variable.
We set $\mathcal{C}=\mathbb{Q}[\hbar]$ and denote
the non-commutative polynomial ring
of two variables $a$ and $b$ over $\mathcal{C}$
by $\mathfrak{H}=\mathcal{C}\langle a, b \rangle$.
For $k \ge 1$ we set
\begin{align*}
    g_{k}=ba^{k}, \qquad e_{k}=b(a+\hbar)a^{k-1}.
\end{align*}
Then the set
\begin{align*}
    \mathcal{A}=\{\hbar b\} \cup \{ba^{k} \mid k \ge 1\}=\{e_{1}-g_{1}\} \cup \{g_{k} \mid k\ge 1\}.
\end{align*}
is algebraically independent.
We denote by $\mathcal{M}$ the set consisting of
$1$ and monomials in $\mathcal{A}$,
and by $\mathcal{M}_{0}$ the subset of $\mathcal{M}$
consisting of the monomials
whose rightmost letter is not $e_{1}-g_{1}$.
In other words, $\mathcal{M}_{0}$ is the set of elements
of the form
\begin{align}
    (e_{1}-g_{1})^{\alpha_{1}} g_{\beta_{1}+1}
    \cdots  (e_{1}-g_{1})^{\alpha_{r}} g_{\beta_{r}+1}
    \label{eq:basis-monomial}
\end{align}
with $r \ge 0$ and $\alpha_{1}, \ldots , \alpha_{r}, \beta_{1}, \ldots, \beta_{r}\ge 0$.
We denote by $\widehat{\mathfrak{H}^{0}}$ (resp. $\widehat{\mathfrak{H}^{1}}$)
the $\mathcal{C}$-submodule of $\mathfrak{H}$ spanned by $\mathcal{M}_{0}$
(resp. $\mathcal{M}$).

The shuffle product $\shp_{\hbar}$ on $\mathfrak{H}$
is the $\mathcal{C}$-bilinear binary operation
uniquely defined by the following properties:
\begin{enumerate}
    \item For any $w \in \mathfrak{H}$,
          it holds that $w \,\shp_{\hbar}\, 1=w$ and $1\,\shp_{\hbar}\,w=w$.
    \item For any $w, w'\in \mathfrak{H}$, it holds that
          \begin{align*}
               &
              wa\,\shp_{\hbar}\,w'a=(wa\,\shp_{\hbar}\, w'+w\,\shp_{\hbar}\, w'a+\hbar w\,\shp_{\hbar}\, w')a, \\
               &
              wb\,\shp_{\hbar}\, w'=w\,\shp_{\hbar}\, w'b=(w\,\shp_{\hbar}\, w')b.
          \end{align*}
\end{enumerate}
The shuffle product $\shp_{\hbar}$ is commutative and associative.
Since $e_{1}-g_{1}=\hbar b$, we have
\begin{align}
    (w(e_{1}-g_{1}))\,\shp_{\hbar}\,w'=w\,\shp_{\hbar}\,(w'(e_{1}-g_{1}))=
    (w\,\shp_{\hbar}\,w')(e_{1}-g_{1})
    \label{eq:e1g1-central}
\end{align}
for any $w, w' \in \mathfrak{H}$.
Consider the formal power series $H(X)$ in $X$ given by
\begin{align}
    H(X)=\sum_{k\ge 1}g_{k}X^{k-1}=g_{1}\frac{1}{1-aX}.
    \label{eq:def-H(X)}
\end{align}
Then it holds that
\begin{align}
     &
    w H(X) \,\shp_{\hbar}\, w' H(Y)
    \label{eq:shuffle-prod-aa} \\
     & =\left\{
    (1+\hbar X)\left(w H(X)\shp_{\hbar}\,w' \right)+
    (1+\hbar Y)\left(w\,\shp_{\hbar}\,w' H(Y)\right)+
    (w\,\shp_{\hbar}\,w')(e_{1}-g_{1}) \right\}
    \nonumber                  \\
     & \times
    H(X+Y+\hbar XY)
    \nonumber
\end{align}
for any $w, w' \in \mathfrak{H}$ as a formal power series in
$X$ and $Y$
(see \cite[Lemma 3.10]{qSMZV}).
Therefore, the $\mathcal{C}$-module $\widehat{\mathfrak{H}^{0}}$
is closed under the shuffle product.

We denote by $\mathfrak{z}$
the $\mathcal{C}$-module spanned by $\mathcal{A}$,
and define the symmetric $\mathcal{C}$-bilinear map
$\circ_{\hbar}: \mathfrak{z}\times \mathfrak{z} \to \mathfrak{z}$ by
\begin{align*}
    (e_{1}-g_{1})\circ_{\hbar}(e_{1}-g_{1})=\hbar (e_{1}-g_{1}), \quad
    (e_{1}-g_{1})\circ_{\hbar} g_{k}=\hbar g_{k}, \quad
    g_{k}\circ_{\hbar} g_{l}=g_{k+l}
\end{align*}
for $k, l \ge 1$.
The harmonic product $\ast_{\hbar}$ is
the $\mathcal{C}$-bilinear binary operation on
$\widehat{\mathfrak{H}^{1}}$
uniquely defined by the following properties:
\begin{enumerate}
    \item For any $w \in \widehat{\mathfrak{H}^{1}}$,
          it holds that $w \ast_{\hbar} 1=w$ and $1\ast_{\hbar} w=w$.
    \item For any $w, w'\in \widehat{\mathfrak{H}^{1}}$ and $u, v \in \mathcal{A}$,
          it holds that $(w u)\ast_{\hbar}(w' v)=(w\ast_{\hbar} w' v)u+(wu\ast_{\hbar} w')v+
              (w\ast_{\hbar} w')(u\circ_{\hbar}v)$.
\end{enumerate}
The harmonic product $\ast_{\hbar}$ is commutative and associative,
and the $\mathcal{C}$-module $\widehat{\mathfrak{H}^{0}}$
is closed under it.

We define the anti-involution $\sigma$ on the scalar extension
$\mathbb{Q}[\hbar, \hbar^{-1}] \otimes_{\mathcal{C}} \mathfrak{H}$ by
\begin{align*}
    \sigma(a)=\hbar b, \qquad \sigma(b)=\hbar^{-1}a
\end{align*}
and $\mathbb{Q}[\hbar, \hbar^{-1}]$-linearity.
Then it holds that
\begin{align}
     &
    \sigma(    (e_{1}-g_{1})^{\alpha_{1}} g_{\beta_{1}+1}
    \cdots  (e_{1}-g_{1})^{\alpha_{r}} g_{\beta_{r}+1} )
    \label{eq:dual-monomial}
    \\
     & =(e_{1}-g_{1})^{\beta_{r}} g_{\alpha_{r}+1}
    \cdots  (e_{1}-g_{1})^{\beta_{1}} g_{\alpha_{1}+1}
    \nonumber
\end{align}
for non-negative integers
$\alpha_{1}, \ldots , \alpha_{r}$
and $\beta_{1}, \ldots , \beta_{r}$.
Hence the map $\sigma$ restricted to
$\widehat{\mathfrak{H}^{0}}$
is a $\mathcal{C}$-linear isomorphism.

\begin{prop}[\cite{Satoh}, Corollary 3.22]
    \label{prop:sigma}
    For any $w, w' \in \widehat{\mathfrak{H}^{0}}$,
    we have
    $w\,\ast_{\hbar}\, w'=\sigma(\sigma(w)\shp_{\hbar}\sigma(w'))$.
\end{prop}


\subsection{A $q$-analogue of multiple zeta values}
\label{subsec:qMZV}

Let $q$ be a real parameter satisfying $0<q<1$.
We endow the real number field $\mathbb{R}$
with $\mathcal{C}$-module structure
such that $\hbar$ acts as multiplication by $1-q$.
For a positive integer $m$,
we define the $\mathcal{C}$-linear map
$F_{q}(\cdot \,|\, m): \mathfrak{z} \to \mathbb{R}$ by
\begin{align*}
    F_{q}(e_{1}-g_{1} \,|\, m)=1-q, \qquad
    F_{q}(g_{k}\,|\, m)=\left(\frac{q^{m}}{[m]}\right)^{k}
\end{align*}
for $k\ge 1$, where $[m]=(1-q^{m})/(1-q)$
is the $q$-integer.
{}From $e_{1}=(e_{1}-g_{1})+g_{1}$ and
$e_{k+1}=g_{k+1}+\hbar g_{k}$, we see that
\begin{align*}
    F_{q}(e_{k}\,|\, m)=\frac{q^{(k-1)m}}{[m]^{k}}
\end{align*}
for $k \ge 1$.

We define the $\mathcal{C}$-linear map
$Z_{q}: \widehat{\mathfrak{H}^{0}} \to \mathbb{R}$ by
$Z_{q}(1)=1$ and
\begin{align}
    Z_{q}(u_{1}\cdots u_{r})=
    \sum_{0<m_{1}<\cdots <m_{r}}
    \prod_{a=1}^{r}F_{q}(u_{a}\,|\,m_{a})
    \label{eq:def-qMZV}
\end{align}
for $u_{a} \in \mathcal{A} \,(1\le a \le r)$
such that $u_{r}\not=e_{1}-g_{1}$.
We call the value $Z_{q}(w)$ with $w \in \widehat{\mathfrak{H}^{0}}$
a \textit{$q$-analogue of MZV} ($q$MZV).

For an admissible index $\bk=(k_{1}, \ldots , k_{r})$,
we set
\begin{align*}
    \zeta_{q}(\bk)=Z_{q}(e_{k_{1}} \cdots e_{k_{r}})=
    \sum_{0<m_{1}<\cdots <m_{r}}
    \frac{q^{(k_{1}-1)m_{1}+\cdots +(k_{r}-1)m_{r}}}
    {[m_{1}]^{k_{1}} \cdots [m_{r}]^{k_{r}}}.
\end{align*}
It is called the \textit{Bradley-Zhao model} of $q$MZV.
In the limit as $q\to 1-0$, it turns into the
multiple zeta value, that is, we have
\begin{align*}
    \lim_{q \to 1-0}\zeta_{q}(\bk)=\zeta(\bk)
\end{align*}
for any admissible index $\bk$.

\subsection{Double shuffle relations}

$q$MZVs satisfy the following relations (see, e.g., \cite{Zhaobook}):

\begin{prop}\label{prop:qMZV-DS}
    For any $w, w' \in \widehat{\mathfrak{H}^{0}}$,
    the following equalities hold.
    \begin{enumerate}
        \item $Z_{q}(w \ast_{\hbar} w')=Z_{q}(w)Z_{q}(w')$ (harmonic product relation)
        \item $Z_{q}(w \,\shp_{\hbar}\, w')=Z_{q}(w)Z_{q}(w')$ (shuffle product relation)
        \item $Z_{q}(\sigma(w))=Z_{q}(w)$ (duality)
    \end{enumerate}
\end{prop}

The above relations (i) and (ii) imply
the following relations:

\begin{cor}[Double shuffle relations of $q$MZVs]
    For any $w, w' \in \widehat{\mathfrak{H}^{0}}$,
    it holds that
    \begin{align*}
        Z_{\omega}(w\,\shp_{\hbar}\,w'-w\ast_{\hbar}w')=0.
    \end{align*}
\end{cor}

\begin{rem}\label{rem:DS}
    By using Proposition \ref{prop:sigma},
    we see that
    the relation (i) in Proposition \ref{prop:qMZV-DS}
    can be derived from (ii), (iii) as follows:
    \begin{align*}
        Z_{q}(w\ast_{\hbar}w')
         & =
        Z_{q}(\sigma(\sigma(w)\,\shp_{\hbar}\,\sigma(w')))=
        Z_{q}(\sigma(w)\,\shp_{\hbar}\,\sigma(w')) \\
         & =
        Z_{q}(\sigma(w))Z_{\omega}(\sigma(w'))=
        Z_{q}(w)Z_{\omega}(w').
    \end{align*}
    Hence the double shuffle relations follow also from
    (ii) and (iii) in Proposition \ref{prop:qMZV-DS}.
\end{rem}

\subsection{Extended double Ohno relations}
\label{subsec:DO-qMZV}

For an admissible index $\bk=(k_{1}, \ldots , k_{r})$
and non-negative integers $m$ and $n$,
we define the \textit{double Ohno sum} $O_{m, n}(\bk)$ by
\begin{align*}
    O_{m, n}(\bk)=
    \sum_{   \substack{
    m_{1}, \ldots , m_{r}, n_{1}, \ldots , n_{r}\ge 0 \\
    m_{1}+\cdots +m_{r}=m                             \\
    n_{1}+\cdots +n_{r}=n}}
    \zeta_{q}(k_{1}+m_{1}+n_{1}, \ldots , k_{r}+m_{r}+n_{r}).
\end{align*}
Then the infinite sum
\begin{align*}
    O(\bk)=\sum_{m, n \ge 0}O_{m, n}(\bk) \xi^{m}\eta^{n}
\end{align*}
absolutely converges in a neighborhood of $(\xi, \eta)=(0, 0)$.

Let $\mathfrak{h}=\mathbb{Q}\langle x, y \rangle$ be
the non-commutative polynomial ring
of two variables $x$ and $y$
over $\mathbb{Q}$.
Set $z_{k}=yx^{k-1}$ for $k\ge 1$ and
$z_{\bk}=z_{k_{1}}\cdots z_{k_{r}}$ for an index $\bk=(k_{1}, \ldots , k_{r})$.
The $\mathbb{Q}$-subalgebra $y\mathfrak{h}x$ of $\mathfrak{h}$
is freely generated by
the set $\{z_{\bk}\mid \hbox{$\bk$ is an admissible index}\}$.
We fix $\xi$ and $\eta$, and define the $\mathbb{Q}$-linear map
$O: y\mathfrak{h}x \longrightarrow \mathbb{C}$
by $O(z_{k_{1}} \cdots z_{k_{r}})=O(\bk)$
for any admissible index $\bk=(k_{1}, \ldots , k_{r})$,
and extend it to the $\mathbb{Q}$-linear map
$O: (y\mathfrak{h}x)[[X]] \longrightarrow \mathbb{C}$ by
\begin{align*}
    O(\sum_{j\ge 0}w_{j}X^{j})=
    \sum_{j\ge 0}O(w_{j})\xi^{j}\eta^{j}
\end{align*}
for $w_{j} \in y\mathfrak{h}x \, (j\ge 0)$.

Denote by $\mathfrak{a}$ the subring of $\mathfrak{h}[[X]]$
generated by $x, y, X$ and $(1+yx X)^{-1}$.
Define the $\mathbb{Q}$-anti-automorphism $\tau$ by
$\tau(X)=X$ and
\begin{align*}
    \tau(x)=(1+yx X)^{-1}y=y(1+xy X)^{-1}, \qquad
    \tau(y)=x(1+yx X)=(1+xy X)x.
\end{align*}

\begin{thm}[Extended double Ohno relations \cite{HSS}]
    \label{thm:extended-DO-qMZV}
    For any $w \in \mathfrak{a}$,
    the infinite sum $O(ywx)$ is absolutely convergent
    in a neighborhood of $(\xi, \eta)=(0,0)$ and
    it holds that
    $O(ywx)=O( y \tau(w) x )$.
\end{thm}

\begin{rem}
    The equality $O(ywx)=O( y \tau(w) x )$ holds as
    a power series in $\xi$ and $\eta$.
    By taking each coefficient of the both hand sides,
    we obtain finite linear relations among
    the double Ohno sums.
    We will prove the same equality for
    a new deformation of MZV,
    then the infinite sum $O(ywx)$ is defined
    as a formal power series in $\xi$ and $\eta$ to
    avoid consideration of convergence.
\end{rem}


\section{$\omega$-deformed multiple zeta value}
\label{sec:def-omegaMZV}

\subsection{Analytic preliminaries}

Hereafter we assume that
\begin{align*}
    0<\omega<2.
\end{align*}
Let $M(\mathbb{C})$ be
the field of meromorphic functions on $\mathbb{C}$.
We endow $M(\mathbb{C})$ with $\mathcal{C}$-module structure
such that $\hbar$ acts as multiplication by $2\pi i \omega$,
that is, $(\hbar f)(t)=2\pi i \omega f(t)$ for $f \in M(\mathbb{C})$.
Then we define the $\mathcal{C}$-linear map
$I_{\omega}: \mathfrak{z} \to M(\mathbb{C})$ by
\begin{align*}
    I_{\omega}(e_{1}-g_{1} \, | \, t)=2\pi i \omega, \qquad
    I_{\omega}(g_{k} \, | \, t)=\left(
    \frac{2\pi i \omega \, e^{2\pi i \omega t}}{1-e^{2\pi i \omega t}}
    \right)^{k}
    =\left(        \frac{2\pi i \omega }{e^{-2\pi i \omega t}-1}
    \right)^{k}
\end{align*}
for $k \ge 1$.
It holds that
\begin{align*}
    I_{\omega}(e_{k}\, | \, t)=(2\pi i \omega)^{k}
    \frac{e^{2\pi i \omega (k-1) t}}{(1-e^{2\pi i \omega t})^{k}}
\end{align*}
for $k \ge 1$.

\begin{prop}\label{prop:integral-convergent}
    Let $(w_{1}, w_{2}, w_{3})$ be a triple of monomials in $\mathcal{A}$
    which belongs to the set
    \begin{align}
        \mathcal{T}=(\mathcal{M}\times \mathcal{M} \times (\mathcal{M}_{0}\setminus\{1\}))
        \cup (\mathcal{M}_{0}\times \mathcal{M}_{0} \times \{1\}).
        \label{eq:def-monomial-T}
    \end{align}
    Set $w_{j}=v_{j,1}\cdots v_{j, r_{j}}$,
    where $v_{j, a} \, (1\le j \le 3, \, 1\le a \le r_{j})$ is an element of $\mathcal{A}$.
    If $w_{j}=1$, set $r_{j}=0$.
    Suppose that $\epsilon$ is a constant satisfying
    $0<\epsilon<\min(1, \omega/(r_{1}+r_{2}+r_{3}))$.
    Then the multiple integral
    \begin{align}
        \mathcal{F}(w_{1}, w_{2}, w_{3})
         & =\prod_{j=1}^{3}\prod_{a=1}^{r_{j}}\int_{{}-\epsilon+i\,\mathbb{R}}
        \frac{dt_{j, a}}{e^{2\pi i t_{j, a}}-1}
        \prod_{j=1}^{2}\prod_{a=1}^{r_{j}}
        I_{\omega}(v_{j, a}\, |\, t_{j, 1}+\cdots +t_{j, a})
        \label{eq:integral-convergent}                                         \\
         & \qquad {}\times
        \prod_{a=1}^{r_{3}}
        I_{\omega}(v_{3, a}\, | \, [t_{1}]+[t_{2}]+t_{3, 1}+\cdots +t_{3, a}),
        \nonumber
    \end{align}
    is absolutely convergent and does not depend on $\epsilon$,
    where $[t_{j}]=\sum_{a=1}^{r_{j}}t_{j, a} \, (j=1, 2)$.
\end{prop}

To prove Proposition \ref{prop:integral-convergent},
we should estimate the integrand of $\mathcal{F}(w_{1}, w_{2}, w_{3})$.
For this purpose, we show a basic inequality.

\begin{lem}\label{lem:estimate-below}
    Set
    \begin{align}
        c(y)=\left\{
        \begin{array}{cl}
            |\sin{y}| & (\hbox{if}\,\, \cos{y}\ge 0) \\
            1         & (\hbox{if}\,\, \cos{y}\le 0)
        \end{array}
        \right.
        \label{eq:def-c(y)}
    \end{align}
    Then
    \begin{align*}
        |e^{x+iy}-1|\ge c(y)e^{\frac{1}{2}(x+|x|)}.
    \end{align*}
    for any $x, y \in \mathbb{R}$.
\end{lem}

\begin{proof}
    We see that $|e^{x+iy}-1|^{2}=(e^{x}-\cos{y})^{2}+\sin^{2}{y}$.
    If $x\le 0$, then $0<e^{x}\le 1$, and hence
    $|e^{x+iy}-1|^{2}\ge c(y)^{2}$.
    Since $|e^{x+iy}-1|=e^{x}|e^{-(x+iy)}-1|$ and $c(-y)=c(y)$,
    it holds that $|e^{x+iy}-1|\ge e^{x}c(y)$ for $x\ge 0$.
\end{proof}

{}From Lemma \ref{lem:estimate-below},
we obtain the following estimate.

\begin{prop}\label{prop:I-estimate}
    Suppose that $0<\rho<\min{(1, 1/\omega)}$
    and set $t=-\rho+iu$ with $u\in \mathbb{R}$.
    Then, it holds that
    \begin{align*}
         &
        \left|\frac{1}{e^{2\pi i t}-1}\right|\le
        \frac{e^{\pi (u-|u|)}}{c(-2\pi \rho)},     \qquad
        |I_{\omega}(e_{1}-g_{1}\,|\,t)|=2\pi \omega, \\
         &
        |I_{\omega}(g_{k}\,|\,t)|\le
        \left(\frac{2\pi \omega}{c(2\pi \omega \rho)}\right)^{k}
        e^{-k \pi \omega (u+|u|)}
        \qquad (k\ge 1),
    \end{align*}
    where $c(y)$ is defined by \eqref{eq:def-c(y)}.
    Hence the function $I_{\omega}(v\,|\,t)$ with $v \in \mathcal{A}$
    is bounded on the line $-\rho+i\,\mathbb{R}$.
\end{prop}

Now we are ready to prove Proposition \ref{prop:integral-convergent}.

\begin{proof}[Proof of Proposition \ref{prop:integral-convergent}]
    Here we prove the proposition in the case where
    $(w_{1}, w_{2}, w_{3})\in
        \mathcal{M}\times \mathcal{M} \times (\mathcal{M}_{0}\setminus\{1\})$.
    The proof for the case where
    $(w_{1}, w_{2}, w_{3})\in
        \mathcal{M}_{0}\times \mathcal{M}_{0} \times \{1\}$
    is similar.

    Denote by $F(t)$ the integrand of
    \eqref{eq:integral-convergent}.
    Since $w_{3}\in \mathcal{M}_{0}\setminus\{1\}$,
    we can set $v_{3, r_{3}}=g_{k}$ with some $k\ge 1$.
    Using Proposition \ref{prop:I-estimate},
    we see that there exists a positive constant
    $C=C(\omega, r , \epsilon)$
    such that
    \begin{align}
        |F(t)|\le C \left( \prod_{j=1}^{3}\prod_{a=1}^{r_{j}}
        \frac{1}{\left|e^{2\pi i t_{j,a}}-1\right|} \right)
        \left|I_{\omega}(g_{k}\,|\,[t_{1}]+[t_{2}]+[t_{3}]) \right|,
        \label{eq:integral-convergent-pf1}
    \end{align}
    on the integral region, where $[t_{3}]=\sum_{a=1}^{r_{3}}t_{3, a}$.
    Set $u_{j,a}=\mathop{\mathrm{Im}}{t_{j, a}}$.
    Using Proposition \ref{prop:I-estimate} again,
    we see that the right hand side of \eqref{eq:integral-convergent-pf1}
    is estimated from above by $C' \exp{(\pi S(u))}$,
    where $C'$ is a positive constant and
    \begin{align*}
        S(u)=\sum_{j=1}^{3}\sum_{a=1}^{r_{j}}(u_{j, a}-|u_{j, a}|)-k\omega
        \left(\sum_{j=1}^{3}\sum_{a=1}^{r_{j}}u_{j, a}+\left|\sum_{j=1}^{3}\sum_{a=1}^{r_{j}}u_{j, a}\right|\right).
    \end{align*}
    Since $k\ge 1$ and $x+|x|\ge 0$ for any $x \in \mathbb{R}$,
    it holds that
    \begin{align*}
        S(u) & \le
        \sum_{j=1}^{3}\sum_{a=1}^{r_{j}}(u_{j, a}-|u_{j, a}|)-\omega
        \left(\sum_{j=1}^{3}\sum_{a=1}^{r_{j}}u_{j, a}+
        \left|\sum_{j=1}^{3}\sum_{a=1}^{r_{j}}u_{j, a}\right|\right) \\
             & \le
        \sum_{j=1}^{3}\sum_{a=1}^{r_{j}}(u_{j, a}-|u_{j, a}|)-\omega
        \sum_{j=1}^{3}\sum_{a=1}^{r_{j}}u_{j, a}
        \le
        (|1-\omega|-1)\sum_{j=1}^{3}\sum_{a=1}^{r_{j}}|u_{j, a}|.
    \end{align*}
    Since $0<\omega<2$, we have $|1-\omega|-1<0$.
    Therefore, the integral \eqref{eq:integral-convergent} is absolutely convergent
    and independent of $\epsilon$.
\end{proof}

\subsection{Definition of $\omega$-deformed multiple zeta value}

\begin{definition}
    We endow the complex number field $\mathbb{C}$ with
    $\mathcal{C}$-module structure such that $\hbar$ acts as multiplication by $2\pi i \omega$.
    We define the $\mathcal{C}$-linear map
    $Z_{\omega}: \widehat{\mathfrak{H}^{0}} \to \mathbb{C}$ by
    $Z_{\omega}(1)=1$ and
    \begin{align*}
        Z_{\omega}(u_{1}\cdots u_{r})=
        \prod_{a=1}^{r}\int_{{}-\epsilon+i\,\mathbb{R}}\frac{dt_{a}}{e^{2\pi i t_{a}}-1}
        \prod_{a=1}^{r}I_{\omega}(u_{a}\,|\, t_{1}+\cdots +t_{a})
    \end{align*}
    for $u_{1}, \ldots , u_{r} \in \mathcal{A}$ such that $u_{r}\not=e_{1}-g_{1}$,
    where $\epsilon$ is a constant satisfying $0<\epsilon<\min(1, 1/r \omega)$.
    Proposition \ref{prop:integral-convergent}
    with $w_{2}=w_{3}=1$ implies that
    the above integral is absolutely convergent and does not depend on $\epsilon$.
    We call the value $Z_{\omega}(w)$ with $w \in \widehat{\mathfrak{H}^{0}}$
    an \textit{$\omega$-deformed multiple zeta value} ($\omega$MZV).
\end{definition}

We give a more explicit formula of $\omega$MZV for
the monomial \eqref{eq:basis-monomial}.

\begin{prop}\label{prop:omegaMZV-formula}
    Let $r$ be a positive integer and
    suppose that $0<\epsilon<\min(1, 1/r \omega)$.
    For any non-negative integers
    $\alpha_{1}, \ldots , \alpha_{r}$ and
    $\beta_{1}, \ldots , \beta_{r}$,
    it holds that
    \begin{align}
         &
        Z_{\omega}((e_{1}-g_{1})^{\alpha_{1}}g_{\beta_{1}+1} \cdots
        (e_{1}-g_{1})^{\alpha_{r}}g_{\beta_{r}+1})
        \label{eq:omegaMZV-formula}                     \\
         & =\prod_{a=1}^{r}\int_{-\epsilon+i\mathbb{R}}
        \frac{dt_{a}}{e^{2\pi i t_{a}}-1}
        \prod_{a=1}^{r}\left\{
        (-2\pi i \omega)^{\alpha_{a}}
        \binom{t_{a}+\alpha_{a}}{\alpha_{a}}
        \left(\frac{2\pi i \omega }{1-e^{-2\pi i \omega (t_{1}+\cdots +t_{a})}}
        \right)^{\beta_{a}+1}
        \right\}.
        \nonumber
    \end{align}
\end{prop}

\begin{proof}
    {}From the definition of $Z_{\omega}$, we have
    \begin{align*}
         &
        Z_{\omega}((e_{1}-g_{1})^{\alpha_{1}}g_{\beta_{1}+1} \cdots
        (e_{1}-g_{1})^{\alpha_{r}}g_{\beta_{r}+1})     \\
         & =(2\pi i \omega)^{\sum_{a=1}^{r}\alpha_{a}}
        \prod_{a=1}^{r}\prod_{j=1}^{\alpha_{a}+1}
        \int_{-\epsilon+i\,\mathbb{R}}
        \frac{ds_{a, j}}{e^{2\pi i s_{a, j}}-1}
        \prod_{a=1}^{r}
        I_{\omega}(g_{\beta_{a}+1}\,|\, [s_{1}]+\cdots+[s_{a}]),
    \end{align*}
    where $[s_{a}]=\sum_{j=1}^{\alpha_{a}+1}s_{a, j}$ for $1\le a \le r$.
    Change the integration variable $s_{a, \alpha_{a}+1}$
    to $t_{a}=[s_{a}]$ for $1\le a \le r$.
    Then the above integral becomes
    \begin{align*}
        \prod_{a=1}^{r}\left(
        \int_{-(\alpha_{a}+1)\epsilon+i\,\mathbb{R}}\!\! dt_{a}
        I_{\omega}(g_{\beta_{a}+1}\,|\, t_{1}+\cdots +t_{a})
        \left(\prod_{j=1}^{\alpha_{a}}
        \int_{-\epsilon+i\,\mathbb{R}}
        \frac{ds_{j}}{e^{2\pi i s_{j}}-1}\right)
        \frac{1}{e^{2\pi i (t_{a}-s_{1}-\cdots-s_{\alpha_{a}})}-1}
        \right).
    \end{align*}
    Now the desired equality follows from
    Lemma \ref{lem:e1-g1-integral} below.
\end{proof}

\begin{lem}\label{lem:e1-g1-integral}
    Let $\alpha$ be a positive integer.
    Suppose that $0<\epsilon<1/(\alpha+1)$ and
    $\mathop{\mathrm{Re}}{t}=-(\alpha+1)\epsilon$.
    Then it holds that
    \begin{align*}
        \left(\prod_{j=1}^{\alpha}\int_{-\epsilon+i\,\mathbb{R}}
        \frac{ds_{j}}{e^{2\pi i s_{j}}-1}\right)
        \frac{1}{e^{2\pi i (t-s_{1}-\cdots -s_{\alpha})}-1}=
        (-1)^{\alpha}\binom{t+\alpha}{\alpha}
        \frac{1}{e^{2\pi i  t}-1}.
    \end{align*}
\end{lem}

\begin{proof}
    We use induction on $\alpha$.
    First, in the case of $\alpha=1$, the left hand side is equal to
    \begin{align*}
         &
        \int_{-\epsilon+i\,\mathbb{R}}\frac{ds}{e^{2\pi i s}-1}
        \frac{(t-s+1)-(t-s)}{e^{2\pi i (t-s)}-1}=
        \left(\int_{-\epsilon-1+i\,\mathbb{R}}-\int_{-\epsilon+\,\mathbb{R}}\right)
        \frac{ds}{e^{2\pi i s}-1}
        \frac{t-s}{e^{2\pi i (t-s)}-1}                               \\
         & =-2\pi i \, \mathrm{Res}_{s=-1} \frac{ds}{e^{2\pi i s}-1}
        \frac{t-s}{e^{2\pi i (t-s)}-1}=-\frac{t+1}{e^{2\pi i t}-1}.
    \end{align*}
    Next, suppose that $\alpha \ge 2$ and $\mathop{\mathrm{Re}}{t}=-(\alpha+1)\epsilon$.
        {}Since $\mathop{\mathrm{Re}}{(t-s_{1})}=-\alpha \epsilon$
    for $s_{1}\in -\epsilon+i\,\mathbb{R}$,
    the induction hypothesis implies that
    \begin{align*}
         &
        \left(\prod_{j=1}^{\alpha}\int_{-\epsilon+i\,\mathbb{R}}
        \frac{ds_{j}}{e^{2\pi i s_{j}}-1}\right)
        \frac{1}{e^{2\pi i (t-s_{1}-\cdots -s_{\alpha})}-1}                \\
         & =(-1)^{\alpha-1}
        \int_{-\epsilon+i\,\mathbb{R}}
        \frac{ds_{1}}{e^{2\pi i s_{1}}-1}
        \binom{t-s_{1}+\alpha-1}{\alpha-1}\frac{1}{e^{2\pi i (t-s_{1})}-1} \\
         & =(-1)^{\alpha-1}
        \int_{-\epsilon+i\,\mathbb{R}}
        \frac{ds_{1}}{e^{2\pi i s_{1}}-1}
        \left(\binom{t-s_{1}+\alpha}{\alpha}-\binom{t-s_{1}+\alpha-1}{\alpha}\right)
        \frac{1}{e^{2\pi i (t-s_{1})}-1}                                   \\
         & =(-1)^{\alpha-1}
        \left(\int_{-\epsilon-1+i\,\mathbb{R}}-\int_{-\epsilon+i\,\mathbb{R}}\right)
        \frac{ds_{1}}{e^{2\pi i s_{1}}-1}
        \binom{t-s_{1}+\alpha-1}{\alpha}\frac{1}{e^{2\pi i (t-s_{1})}-1}.
    \end{align*}
    Since $\binom{t-s_{1}+\alpha-1}{\alpha}$ is zero at $s_{1}=t$,
    it is equal to
    \begin{align*}
        (-1)^{\alpha}2\pi i \, \mathrm{Res}_{s_{1}=-1}
        \frac{ds_{1}}{e^{2\pi i s_{1}}-1}
        \binom{t-s_{1}+\alpha-1}{\alpha}\frac{1}{e^{2\pi i (t-s_{1})}-1}=
        (-1)^{\alpha}\binom{t+\alpha}{\alpha}
        \frac{1}{e^{2\pi i t}-1}.
    \end{align*}
    This completes the proof.
\end{proof}

\begin{rem}
    Suppose that $\epsilon>0$ is small enough.
    Then, using Proposition \ref{prop:omegaMZV-formula},
    we see that
    the infinite sum
    \begin{align*}
        R_{\omega}\begin{pmatrix}x_{1}, \ldots , x_{r} \\ y_{1}, \ldots , y_{r}
        \end{pmatrix}
         & =\sum_{\substack{\alpha_{1}, \ldots , \alpha_{r} \ge 0 \\
                \beta_{1}, \ldots , \beta_{r} \ge 0}}
        Z_{\omega}((e_{1}-g_{1})^{\alpha_{1}}g_{\beta_{1}+1} \cdots
        (e_{1}-g_{1})^{\alpha_{r}}g_{\beta_{r}+1})                \\
         & \qquad {}\times
        \prod_{a=1}^{r}
        \left(\frac{e^{2\pi i \omega y_{a}}-1}{2\pi i \omega }\right)^{\alpha_{a}}
        \left(\frac{e^{2\pi i \omega x_{a}}-1}{2\pi i \omega }\right)^{\beta_{a}}
    \end{align*}
    converges absolutely
    and defines a holomorphic function
    of $x_{1}, \ldots , x_{r}, y_{1}, \ldots , y_{r}$
    in a neighborhood of the origin of $\mathbb{C}^{2r}$.
        {}From Proposition \ref{prop:omegaMZV-formula}, we see that
    \begin{align}
        R_{\omega}\begin{pmatrix}x_{1}, \ldots , x_{r} \\ y_{1}, \ldots , y_{r}
        \end{pmatrix}=
        (2\pi i \omega)^{r}
        e^{-2\pi i \omega \sum_{a=1}^{r}y_{a}}
        \prod_{a=1}^{r}\int_{-\epsilon+i\,\mathbb{R}}
        \frac{dt_{a}}{e^{2\pi i t_{a}}-1}
        \prod_{a=1}^{r}
        \frac{e^{2\pi i \omega (t_{1}+\cdots +t_{a}-y_{a}t_{a})}}
        {1-e^{2\pi i \omega (x_{a}+t_{1}+\cdots +t_{a})}}.
        \label{eq:omegaMZV-generating-function2}
    \end{align}
    The contour $-\epsilon+i\,\mathbb{R}$ of the integration variable $t_{a}$ in
    \eqref{eq:omegaMZV-generating-function2}
    separates the poles of the integrand at
    $\mathbb{Z}_{\ge 0}$ and
    ${}-x_{k}-\sum_{\substack{1\le j \le k \\ j\not=a}}t_{j}+
        (1/\omega)\mathbb{Z}_{\ge 0} \,\, (k\ge a)$
    {}from those at
    $\mathbb{Z}_{<0}$ and
    ${}-x_{k}-\sum_{\substack{1\le j\le k \\ j\not=a}}t_{j}+
        (1/\omega)\mathbb{Z}_{<0} \,\, (k\ge a)$.

    Any function $f(x_{1}, \ldots , x_{r}, y_{1}, \ldots , y_{r})$
    which is holomorphic in a neighborhood of
    the origin of $\mathbb{C}^{2r}$
    is uniquely expanded into a power series of
    $(e^{2\pi i \omega x_{a}}-1)/2\pi i \omega$ and
    $(e^{2\pi i \omega y_{a}}-1)/2\pi i \omega \, (1\le a \le r)$ since
    their derivatives do not vanish at $x_{a}=0$ or $y_{a}=0$.
    Therefore, $R_{\omega}$ is a generating function of $\omega$MZVs.

    In the case of $\omega=1$,
    the generating function $R_{1}$ is uniquely determined from
    the initial formula
    \begin{align*}
        R_{1}\begin{pmatrix}x_{1} \\ y_{1}
        \end{pmatrix}=2\pi i \,
        \frac{1-e^{2\pi i x_{1}y_{1}}}{(1-e^{2\pi i x_{1}})(1-e^{2\pi i y_{1}})}
    \end{align*}
    and the recurrence relation
    \begin{align*}
         &
        R_{1}\begin{pmatrix}x_{1}, \ldots , x_{r} \\ y_{1}, \ldots , y_{r} \end{pmatrix}=
        \frac{2\pi i}{1-e^{2\pi i y_{r}}}\frac{1}{e^{2\pi i x_{r-1}}-e^{2\pi i x_{r}}} \\
         & \times
        \biggl\{
        R_{1}\begin{pmatrix}x_{1}, \ldots , x_{r-2}, x_{r-1} \\ y_{1}, \ldots , y_{r-2}, y_{r-1} \end{pmatrix}-
        R_{1}\begin{pmatrix}x_{1}, \ldots , x_{r-2}, x_{r} \\ y_{1}, \ldots , y_{r-2}, y_{r-1} \end{pmatrix}                                                \\
         & -e^{2\pi i x_{r}y_{r}}
        \left(
        R_{1}\begin{pmatrix}x_{1}, \ldots , x_{r-2}, x_{r-1} \\ y_{1}-y_{r}, \ldots , y_{r-2}-y_{r}, y_{r-1}-y_{r} \end{pmatrix}-
        R_{1}\begin{pmatrix}x_{1}, \ldots , x_{r-2}, x_{r} \\ y_{1}-y_{r}, \ldots , y_{r-2}-y_{r}, y_{r-1}-y_{r} \end{pmatrix}
        \right)
        \biggr\}
    \end{align*}
    for $r \ge 2$,
    which are derived from the formula \eqref{eq:omegaMZV-generating-function2}
    though we omit the details.
    Therefore,
    any $\omega$MZV with $\omega=1$ is a polynomial in $2\pi i$
    with rational coefficients.
\end{rem}

\subsection{Limit as $\omega \to +0$}

Here we summarize some results about
the limit of $\omega$MZVs as $\omega \to +0$.
The following two propositions are proved
in Appendix \ref{sec:app1}.

\begin{prop}\label{prop:omega-MZV-limit1}
    Let $v$ be the monomial \eqref{eq:basis-monomial} with $r\ge 1$
    and $\alpha_{1}, \ldots , \alpha_{r}, \beta_{1}, \ldots, \beta_{r}\ge 0$.
    \begin{enumerate}
        \item As $\omega \to +0$, it holds that $Z_{\omega}(v)=O((-\log{\omega})^{r})$.
              Hence $\lim_{\omega \to +0}Z_{\omega}(\hbar w)=0$
              for any $w \in \widehat{\mathfrak{H}^{0}}$.
        \item If there exist $s$ and $t$ such that $1\le s\le t\le r$,
              $\alpha_{s}\ge 1$ and
              $\beta_{t}\ge 1$, then $\lim_{\omega \to +0}Z_{\omega}(v)=0$.
    \end{enumerate}
\end{prop}

\begin{prop}\label{prop:limit-Zg}
    Let $(k_{1}, \ldots , k_{r})$ be an admissible index.
    Then it holds that
    \begin{align*}
        \lim_{\omega \to +0}Z_{\omega}(g_{k_{1}} \cdots g_{k_{r}})=
        \zeta(k_{1}, \ldots , k_{r}).
    \end{align*}
\end{prop}

\begin{rem}
    The $q$MZV defined by \eqref{eq:def-qMZV} has
    similar properties.
    See \cite[Proposition 3.4]{qSMZV}.
\end{rem}

\begin{cor}\label{cor:limit-Ze}
    Let $(k_{1}, \ldots , k_{r})$ be an admissible index.
    Then it holds that
    \begin{align*}
        \lim_{\omega \to +0}Z_{\omega}(e_{k_{1}} \cdots e_{k_{r}})=
        \zeta(k_{1}, \ldots , k_{r}).
    \end{align*}
\end{cor}

\begin{proof}
    It holds that $e_{k}=g_{k}+\hbar g_{k-1}$ for $k\ge 2$
    and $e_{1}=(e_{1}-g_{1})+g_{1}$.
    Since $k_{r} \ge 2$, we see that
    $e_{k_{1}}\cdots e_{k_{r-1}}e_{k_{r}}-e_{k_{1}}\cdots e_{k_{r-1}}g_{k_{r}}$
    belongs to $\hbar \widehat{\mathfrak{H}^{0}}$.
    Moreover,
    $e_{k_{1}}\cdots e_{k_{r-1}}g_{k_{r}}-g_{k_{1}}\cdots g_{k_{r-1}}g_{k_{r}}$
    is a $\mathcal{C}$-linear combination
    of monomials \eqref{eq:basis-monomial} with
    $\alpha_{s}\ge 1$ for some $1\le s \le r$
    and $\beta_{r}=k_{r}-1\ge 1$
    modulo $\hbar \widehat{\mathfrak{H}^{0}}$.
    Hence, from Proposition \ref{prop:omega-MZV-limit1}
    and Proposition \ref{prop:limit-Zg},
    we see that
    \begin{align*}
        \lim_{\omega \to +0}Z_{\omega}(e_{k_{1}} \cdots e_{k_{r}})=
        \lim_{\omega \to +0}Z_{\omega}(g_{k_{1}} \cdots g_{k_{r}})=
        \zeta(k_{1}, \ldots , k_{r}).
    \end{align*}
    This completes the proof.
\end{proof}

For an admissible index $\bk=(k_{1}, \ldots, k_{r})$,
we set
\begin{align}
    \zeta_{\omega}(\bk)=Z_{\omega}(e_{k_{1}}\cdots e_{k_{r}})=
    (2\pi i \omega)^{|\bk|}
    \prod_{a=1}^{r}\int_{-\epsilon+i\,\mathbb{R}}\frac{dt_{a}}{e^{2\pi i t_{a}}-1}
    \prod_{a=1}^{r}
    \frac{e^{2\pi i \omega (k_{a}-1)(t_{1}+\cdots +t_{a})}}
    {(1-e^{2\pi i \omega (t_{1}+\cdots +t_{a})})^{k_{a}}},
    \label{eq:omega-MZV-BZ}
\end{align}
where $|\bk|=\sum_{a=1}^{r}k_{a}$ is the weight of $\bk$.
Corollary \ref{cor:limit-Ze} shows that
$\zeta_{\omega}(\bk)\to \zeta(\bk)$ as $\omega \to +0$.


\section{Double shuffle relations}
\label{sec:DS-omegaMZV}

\subsection{Shuffle product relations}

\begin{prop}\label{prop:shuffle-prod}
    Let $\mathfrak{G}$ be the $\mathcal{C}$-submodule of
    the tensor product $\mathfrak{H}^{\otimes 3}$
    spanned by the set
    $\{w_{1}\otimes w_{2}\otimes w_{3}\mid (w_{1}, w_{2}, w_{3})\in \mathcal{T}\}$,
    where $\mathcal{T}$ is the set defined by \eqref{eq:def-monomial-T}.
    Define the $\mathcal{C}$-linear map
    $\mathcal{F}_{\omega}: \mathfrak{G} \to \mathbb{C}$ by
    \eqref{eq:integral-convergent} and $\mathcal{C}$-linearity.
    We also define the $\mathcal{C}$-linear map
    $\mathcal{S}: \mathfrak{G} \to \widehat{\mathfrak{H}^{0}}$ by
    $\mathcal{S}(w_{1}\otimes w_{2} \otimes w_{3})=(w_{1}\,\shp_{\hbar}\,w_{2})w_{3}$
    for $(w_{1}, w_{2}, w_{3})\in \mathcal{T}$.
    Then it holds that
    $Z_{\omega}(\mathcal{S}(w))=\mathcal{F}_{\omega}(w)$
    for any $w \in \mathfrak{G}$.
\end{prop}

\begin{proof}
    It is enough to show the proposition
    in the case of $w=w_{1}\otimes w_{2}\otimes w_{3}$
    with $(w_{1}, w_{2}, w_{3})\in \mathcal{T}$.
    Set $w_{j}=v_{j,1} \cdots v_{j,r_{j}} \, (1\le j \le 3)$, where
    $v_{j,k}\in \mathcal{A}$.
    We prove $Z_{\omega}(\mathcal{S}(w))=\mathcal{F}_{\omega}(w)$
    by induction on $r_{1}+r_{2}$.
    If $r_{1}=0$ or $r_{2}=0$, it follows from the definition of
    $\mathcal{F}_{\omega}$.
    Suppose that $r_{1}\ge 1$ and $r_{2}\ge 1$.
    If $v_{1, r_{1}}$ or $v_{2, r_{2}}$ is equal to $e_{1}-g_{1}$,
    then the desired equality is trivial since
    $I_{\omega}(e_{1}-g_{1}\,|\, t)$ does not depend on $t$.
    Suppose that
    $v_{1, r_{1}}$ and $v_{2, r_{2}}$ belong to the set $\{g_{k}\}_{k\ge 1}$.
    Set
    \begin{align*}
        H(t, X)=\sum_{k\ge 1}I_{\omega}(g_{k}\,|\,t)X^{k-1}.
    \end{align*}
    If $|X|$ is sufficiently small, it holds that
    \begin{align*}
        H(t, X)=\frac{2\pi i \omega}{e^{-2\pi i \omega t}-1-2\pi i \omega X}.
    \end{align*}
    Using this expression we see that
    \begin{align*}
        H(t, X)H(s, Y)
         & =\left(
        (1+2\pi i \omega X)H(t, X)+(1+2\pi i \omega Y)H(s, Y)+2\pi i \omega
        \right)                                 \\
         & \times H(t+s, X+Y+2\pi i \omega XY).
    \end{align*}
    Comparing it with \eqref{eq:shuffle-prod-aa},
    we see that the desired equality follows from the induction hypothesis.
\end{proof}

\begin{thm}[Shuffle product relations of $\omega$MZVs]\label{thm:shuffle-prod}
    For any $w, w' \in \widehat{\mathfrak{H}^{0}}$,
    it holds that
    \begin{align*}
        Z_{\omega}(w\,\shp_{\hbar}\,w')=Z_{\omega}(w)Z_{\omega}(w').
    \end{align*}
\end{thm}

\begin{proof}
    Since $w\otimes w' \otimes 1 \in \mathfrak{G}$,
    Proposition \ref{prop:shuffle-prod} and
    the definition of $\mathcal{F}_{\omega}$ imply that
    \begin{align*}
        Z_{\omega}(w\,\shp_{\hbar}\,w')=
        Z_{\omega}(\mathcal{S}(w\otimes w'\otimes 1))=
        \mathcal{F}_{\omega}(w \otimes w' \otimes 1)=
        Z_{\omega}(w)Z_{\omega}(w').
    \end{align*}
    This completes the proof.
\end{proof}

\subsection{Duality and double shuffle relations}

\begin{thm}[Duality of $\omega$MZVs] \label{thm:duality}
    For any $w \in \widehat{\mathfrak{H}^{0}}$, it holds that
    \begin{align*}
        Z_{\omega}(\sigma(w))=Z_{\omega}(w).
    \end{align*}
\end{thm}

In the proof of Theorem \ref{thm:duality} we use the following equality.

\begin{lem}\label{lem:integral-duality}
    Suppose that $0<\epsilon<1$ and $0<\mathop{\mathrm{Im}}{\alpha}<2\pi$.
    Then we have
    \begin{align*}
        \int_{-\epsilon+i\,\mathbb{R}}dt\,\frac{e^{\alpha t}}{e^{2\pi i t}-1}=
        \frac{1}{e^{\alpha}-1}.
    \end{align*}
\end{lem}

\begin{proof}
    Since $0<\mathop{\mathrm{Im}}{\alpha}<2\pi$,
    the integral in the left hand side is absolutely convergent.
    The desired formula is derived as follows.
    \begin{align*}
        \int_{-\epsilon+i\,\mathbb{R}}dt\,\frac{e^{\alpha t}}{e^{2\pi i t}-1}
         & =
        \frac{1}{e^{\alpha}-1}
        \int_{-\epsilon+i\,\mathbb{R}}dt\,\frac{e^{\alpha (t+1)}-e^{\alpha t}}
        {e^{2\pi i t}-1}                        \\
         & =
        \frac{1}{e^{\alpha}-1}
        \left(\int_{-\epsilon+1+i\,\mathbb{R}}-\int_{-\epsilon+i\,\mathbb{R}}\right)
        dt\,\frac{e^{\alpha t}}{e^{2\pi i t}-1} \\
         & =
        \frac{2\pi i}{e^{\alpha}-1}
        \mathrm{Res}_{t=0}
        \frac{e^{\alpha t}}{e^{2\pi i t}-1}dt=
        \frac{1}{e^{\alpha}-1}.
    \end{align*}
\end{proof}

\begin{proof}[Proof of Theorem \ref{thm:duality}]
    It is enough to prove the equality
    $Z_{\omega}(w)=Z_{\omega}(\sigma(w))$
    in the case where
    $w$ is the monomial \eqref{eq:basis-monomial}
    with $r\ge 1$ and $\alpha_{1}, \ldots , \alpha_{r}, \beta_{1}, \ldots , \beta_{r}\ge 0$.
    Note that $\sigma(w)$ is given by \eqref{eq:dual-monomial}.
    {}From the definition we have
    \begin{align*}
        Z_{\omega}(w)=(2\pi i \omega)^{\sum_{j=1}^{r}\alpha_{j}}
        \prod_{j=1}^{r}\prod_{a=1}^{\alpha_{j}+1}
        \int_{-\epsilon+i\,\mathbb{R}}\frac{dt_{j, a}}{e^{2\pi i t_{j,a}}-1}
        \prod_{j=1}^{r}
        I_{\omega}(g_{\beta_{j}+1}\,|\,[t_{1}]+\cdots +[t_{j}]),
    \end{align*}
    where $[t_{j}]=\sum_{a=1}^{\alpha_{a}+1}t_{j,a} \, (1\le j \le r)$.
    Using Lemma \ref{lem:integral-duality} we see that
    \begin{align*}
        I_{\omega}(g_{\beta+1}\,|\,t)=(2\pi i \omega)^{\beta+1}
        \left(\prod_{b=1}^{\beta+1}\int_{-\epsilon+i\,\mathbb{R}}
        \frac{ds_{b}}{e^{2\pi i s_{b}}-1}\right)
        e^{-2\pi i \omega t [s]}
    \end{align*}
    for $\beta\ge 0$, where $[s]=\sum_{b=1}^{\beta+1}s_{b}$.
    Hence it holds that
    \begin{align}
        Z_{\omega}(w)
         & =(2\pi i \omega)^{\sum_{j=1}^{r}(\alpha_{j}+\beta_{j}+1)}
        \prod_{j=1}^{r}
        \left(\prod_{a=1}^{\alpha_{j}+1}
        \int_{-\epsilon+i\,\mathbb{R}}\frac{dt_{j, a}}{e^{2\pi i t_{j,a}}-1}
        \prod_{b=1}^{\beta_{j}+1}
        \int_{-\epsilon+i\,\mathbb{R}}\frac{ds_{j, b}}{e^{2\pi i s_{j,b}}-1}
        \right)
        \label{eq:duality-proof1}                                    \\
         & \qquad \quad {}\times
        e^{-2 \pi i \omega \sum_{1\le a\le b \le r}[t_{a}][s_{b}]},
        \nonumber
    \end{align}
    where $[s_{b}]=\sum_{j=1}^{\beta_{b}+1}s_{j, b}$.
    Now change the integration variables
    $t_{j, a}\to s_{r+1-j, a}$ and $s_{j, b}\to t_{r+1-j, b}$,
    and we obtain the integral \eqref{eq:duality-proof1} with
    $\alpha_{j}$ and $\beta_{j}$ replaced by
    $\beta_{r+1-j}$ and $\alpha_{r+1-j}$, respectively.
    Thus we see that $Z_{\omega}(w)=Z_{\omega}(\sigma(w))$.
\end{proof}

As discussed in Remark \ref{rem:DS},
we obtain the harmonic product relations
and the double shuffle relations for $\omega$MZVs from
Theorem \ref{thm:shuffle-prod} and Theorem \ref{thm:duality}:

\begin{thm}[Harmonic product relations of $\omega$MZVs] \label{thm:harmonic-prod}
    For any $w, w' \in \widehat{\mathfrak{H}^{0}}$,
    it holds that
    \begin{align*}
        Z_{\omega}(w\ast_{\hbar}w')=Z_{\omega}(w)Z_{\omega}(w').
    \end{align*}
\end{thm}

\begin{thm}[Double shuffle relations of $\omega$MZVs]
    \label{thm:omegaMZV-DS}
    For any $w, w' \in \widehat{\mathfrak{H}^{0}}$,
    it holds that
    \begin{align*}
        Z_{\omega}(w\,\shp_{\hbar}\,w'-w\ast_{\hbar}w')=0.
    \end{align*}
\end{thm}


\section{Extended double Ohno relations}\label{sec:extended-DO}

\subsection{Algebraic formulation}

For an admissible index $\bk=(k_{1}, \ldots , k_{r})$
and non-negative integers $m$ and $n$,
we define the double Ohno sum of $\omega$MZVs by
\begin{align*}
    O_{m, n}(\bk)=\sum_{
    \substack{m_{1}, \ldots , m_{r}, n_{1}, \ldots , n_{r}\ge 0 \\
            m_{1}+\cdots+m_{r}=m
    \\ n_{1}+\cdots+n_{r}=n}}
    \zeta_{\omega}(k_{1}+m_{1}+n_{1}, \cdots k_{r}+m_{r}+n_{r}),
\end{align*}
where $\zeta_{\omega}(\bk)$ is given by \eqref{eq:omega-MZV-BZ}.
Recall that $\mathfrak{h}=\mathbb{Q}\langle x, y \rangle$
is the non-commutative polynomial ring,
$z_{k}=yx^{k-1}$ for $k\ge 1$ and
$z_{\bk}=z_{k_{1}} \cdots z_{k_{r}}$
for an index $\bk=(k_{1}, \ldots , k_{r})$.
We define the formal power series
$\Omega(\bk \, | \, \xi, \eta)$ in $\xi$ and $\eta$ by
\begin{align}
    \Omega(\bk \, | \, \xi, \eta)=\sum_{m, n\ge 0}
    O_{m, n}(\bk)\xi^{m}\eta^{n},
    \label{eq:def-Omega-monomial}
\end{align}
and define the $\mathbb{Q}$-linear map
$\Omega(\cdot \,|\,\xi,\eta): y\mathfrak{h}x \longrightarrow \mathbb{C}[[\xi, \eta]]$
by $\Omega(z_{\bk}\,|\,\xi,\eta)=\Omega(\bk\,|\,\xi,\eta)$
for any admissible index $\bk$.

For a non-negative integer $n$, we define the ideal $I_{n}$ of
$\mathbb{C}[[\xi,\eta]]$ by
$I_{n}=\{\sum_{i+j\ge n}c_{ij}\xi^{i}\eta^{j} \mid c_{ij}\in\mathbb{C}\}$.
We endow the ring $\mathbb{C}[[\xi,\eta]]$ with
the topology such that the ideals $\{I_{n}\}_{n\ge 0}$
form a fundamental system of neighborhoods of zero.
Then we extend the map $\Omega(\cdot \,|\,\xi,\eta)$ to the $\mathbb{Q}$-linear map
on the formal power series ring $(y\mathfrak{h}x)[[X]]$ by
\begin{align*}
    \Omega({\textstyle \sum_{j\ge 0}w_{j}X^{j}} \, |\,  \xi, \eta)=
    \sum_{j\ge 0}\Omega(w_{j} \, | \,  \xi, \eta)\xi^{j}\eta^{j},
\end{align*}
where $w_{j} \in y\mathfrak{h}x \, (j\ge 0)$.

\begin{thm}[Extended double Ohno relations of $\omega$MZVs]
    \label{thm:extended-double-Ohno}
    For any $w \in \mathfrak{a}$, it holds that
    \begin{align*}
        \Omega(ywx \, | \, \xi, \eta)=\Omega( y \tau(w) x \, | \, \xi, \eta)
    \end{align*}
    (see Section \ref{subsec:DO-qMZV} for the definition of $\mathfrak{a}$ and $\tau$).
\end{thm}

We prove Theorem \ref{thm:extended-double-Ohno}
in the rest of this section.
For this purpose we consider a generating function of
the double Ohno sums.

For $k\ge 1$, we set
\begin{align}
    J_{k}(t\,|\,\lambda, \mu)=
    \frac{e^{2\pi i \omega (k-1)t}}
    {(1-e^{2\pi i \omega (\lambda+t)})(1-e^{2\pi i \omega (\mu+t)})(1-e^{2\pi i \omega t})^{k-2}}.
    \label{eq:def-Jk}
\end{align}

\begin{prop}\label{prop:Ohno-function-integral}
    Let $\bk=(k_{1}, \ldots , k_{r})$ be an admissible index.
    Suppose that $0<\epsilon<1/2r\omega$.
    Then the infinite sum
    \begin{align*}
        O(\bk \,|\, \lambda, \mu)=\sum_{m, n\ge 0}
        O_{m, n}(\bk)
        \left(\frac{e^{2\pi i \omega \lambda}-1}{2\pi i \omega}\right)^{m}
        \left(\frac{e^{2\pi i \omega \mu}-1}{2\pi i \omega}\right)^{n}
    \end{align*}
    defines a holomorphic function
    in the region
    \begin{align}
        \{(\lambda, \mu) \in \mathbb{C}^{2}\mid
        |\lambda|<\epsilon/3\pi, |\mu|<\epsilon/3\pi \}
        \label{eq:Ohno-region}
    \end{align}
    and it holds that
    \begin{align}
        O(\bk \,|\, \lambda, \mu)=(2\pi i \omega)^{|\bk|}
        \prod_{a=1}^{r}\int_{-\epsilon+i\,\mathbb{R}}\frac{dt_{a}}{e^{2\pi i t_{a}}-1}
        \prod_{a=1}^{r}J_{k_{a}}(t_{1}+\cdots +t_{a} \,|\, \lambda, \mu).
        \label{eq:ohno-integral}
    \end{align}
\end{prop}

See Appendix \ref{sec:app2} for the proof of
Proposition \ref{prop:Ohno-function-integral}.

\subsection{Hyperbolic gamma function}

In the proof of Theorem \ref{thm:extended-double-Ohno},
we use the hyperbolic gamma function due to Ruijsenaars \cite{R1}.
Here we summarize some of its properties.
See \cite[Appendix A]{R2} for details.

The hyperbolic gamma function $G(z\,|\,\omega_{1}, \omega_{2})$ has
two complex parameters $\omega_{1}, \omega_{2}$
whose real parts are positive.
In this paper we use it with $(\omega_{1}, \omega_{2})=(1, 1/\omega)$
and write $G(z)=G(z\,|\,1, 1/\omega)$ for simplicity.
We set
\begin{align*}
    \underline{\omega}=\frac{1}{2}\left(1+\frac{1}{\omega}\right).
\end{align*}
The function $G(z)$ is a meromorphic function on $\mathbb{C}$ satisfying
\begin{align}
    G(z)G(-z)=1
    \label{eq:G-inverse}
\end{align}
and
\begin{align}
     &
    G(z+i)=-2i
    \sinh{(\pi \omega (z+i\underline{\omega}))}G(z),
    \label{eq:G-property}
    \\
     &
    G(z+i/\omega)=-2i
    \sinh{(\pi (z+i\underline{\omega}))}G(z).
    \nonumber
\end{align}
The poles of $G(z)$ are at
$-i\underline{\omega}+i\mathbb{Z}_{\le 0}+
    (i/\omega)\mathbb{Z}_{\le 0}$ and
the zeroes are at
$i\underline{\omega}+i\mathbb{Z}_{\ge 0}+
    (i/\omega)\mathbb{Z}_{\ge 0}$.
In the region $|\mathrm{Im}z|<\underline{\omega}$,
it holds that
\begin{align*}
    G(z)=\exp{
        \left(i \int_{0}^{\infty}\frac{dt}{t}
        \left( \frac{\sin{2\omega t z}}{2\sinh{\omega t}\sinh{t}}-
        \frac{z}{t}
        \right)
        \right)
    }.
\end{align*}

Let $K$ be a compact subset of $\mathbb{R}$.
Then we have
\begin{align}
    G(z) \exp{(\frac{\pi i \omega}{2}z^{2}+\frac{\pi i}{24}(\omega+\frac{1}{\omega}))}
    \to 1 \qquad
    (\mathop{\mathrm{Im}}{z}\in K, \, \mathop{\mathrm{Re}}{z} \to +\infty).
    \label{eq:G-limit}
\end{align}
There exist positive constants $M$ and $M'$ which depend on
$\omega$ and $K$ such that
\begin{align}
    |G(z)|
     & \le
    M\left|\exp{(\mp \frac{\pi i \omega}{2}z^{2})} \right|,
    \label{eq:G-estimate0} \\
    \left|\frac{G'(z)}{G(z)}\right|
     & \le
    M'(|z|^{2}+1)
    \label{eq:G'/G-estimate}
\end{align}
if $\pm \mathrm{Re}z>\max(1, 1/\omega)$ and $\mathrm{Im}z\in K$.
Hence, if $K$ is a closed interval contained in
$(-\underline{\omega}, \underline{\omega})$,
we can take a positive constant $M$ so that
\begin{align}
    |G(z)|\le M e^{\pi \omega |\mathop{\mathrm{Re}{z}}| \mathop{\mathrm{Im}}{z}}
    \label{eq:G-estimate}
\end{align}
in the strip
$\{z \in \mathbb{C} \mid \mathop{\mathrm{Im}}{z}\in K\}$.

\subsection{Connected integral}

For indices $\bk=(k_{1}, \ldots, k_{r})$ and
$\bl=(l_{1}, \ldots , l_{s})$,
we set
\begin{align}
     & \label{eq:def-connector}
    \mathcal{I}(\bk, \bl \, |\, \lambda, \mu)=
    (2\pi i \omega)^{|\bk|+|\bl|}
    \prod_{a=1}^{r}\int_{-\epsilon+i\,\mathbb{R}}\frac{dt_{a}}{e^{2\pi i t_{a}}-1}
    \prod_{b=1}^{s}\int_{-\epsilon+i\,\mathbb{R}}\frac{du_{b}}{e^{2\pi i u_{b}}-1} \\
     & \nonumber
    \times
    \prod_{a=1}^{r-1}J_{k_{a}}(t_{1}+\cdots +t_{a} \,|\,\lambda, \mu)
    \prod_{b=1}^{s-1}J_{l_{b}}(u_{1}+\cdots +u_{b} \,|\,\lambda, \mu)              \\
     & \nonumber
    \times
    \left(\frac{e^{2\pi i \omega [t]}}{1-e^{2\pi i \omega [t]}}\right)^{k_{r}-1}
    \left(\frac{e^{2\pi i \omega [u]}}{1-e^{2\pi i \omega [u]}}\right)^{l_{s}-1}
    \Theta([t], [u], \lambda, \mu),
\end{align}
where $J_{k}(t\,|\,\lambda, \mu)$ is defined by \eqref{eq:def-Jk},
$[t]=\sum_{a=1}^{r}t_{a}, [u]=\sum_{b=1}^{s}u_{b}$, and
the function $\Theta(t, u, \lambda, \mu)$ is defined by
\begin{align*}
    \Theta(t, u, \lambda, \mu) & =
    \frac{e^{-\pi i \omega tu}}{1-e^{2\pi i \omega (t+u+\lambda+\mu)}} \\
                               & \times
    \frac{G(i(\underline{\omega}+t))G(i(\underline{\omega}+u))
        G(i(\underline{\omega}+t+u+\lambda+\mu))}
    {G(i(\underline{\omega}+t+\lambda))G(i(\underline{\omega}+t+\mu))
        G(i(\underline{\omega}+u+\lambda))G(i(\underline{\omega}+u+\mu))}.
\end{align*}
Note that
$\mathcal{I}(\bk, \bl\,|\,\lambda, \mu)=\mathcal{I}(\bl, \bk\,|\,\lambda, \mu)$
for any indices $\bk$ and $\bl$
since $\Theta(t, u, \lambda, \mu)$ is symmetric in
$t$ and $u$.

\begin{prop}\label{prop:connector-convergence}
    Let $\bk=(k_{1}, \ldots , k_{r}), \bl=(l_{1}, \ldots , l_{s})$ be indices
    and suppose that
    $0<(r+s+2)\epsilon<\min{(1, 1/\omega)}$.
    Then the integral $\mathcal{I}(\bk, \bl \, |\, \lambda, \mu)$ is
    absolutely convergent and defines a holomorphic function
    on the region
    $B=\left\{ (\lambda, \mu) \in \mathbb{C}^{2}\mid
        |\lambda|<\epsilon, |\mu|<\epsilon \right\}$.
\end{prop}

See Appendix \ref{sec:app2} for the proof of
Proposition \ref{prop:connector-convergence}.

For an index $\bk=(k_{1}, \ldots , k_{r})$, we set
\begin{align*}
    \bk_{\uparrow}=(k_{1}, \ldots , k_{r-1}, k_{r}+1).
\end{align*}
Note that $\bk_{\uparrow}$ is admissible for any index $\bk$.
The integral $\mathcal{I}(\bk,\bl\,|\,\lambda,\mu)$ is related to
the generating function $O(\bk\,|\,\lambda,\mu)$ of the double Ohno sums
as follows.

\begin{thm}\label{thm:DO-initial}
    For any index $\bk$, it holds that
    \begin{align*}
        \mathcal{I}(\bk, \{1\}\,|\,\lambda, \mu)=
        d(\lambda, \mu) O(\bk_{\uparrow}\,|\,\lambda, \mu),
    \end{align*}
    where
    \begin{align}
        d(\lambda, \eta)=\frac{i}{\sqrt{\omega}}e^{\pi i \omega(\lambda+\mu-\lambda\mu)}
        G(i(\underline{\omega}-\lambda-1/\omega))
        G(i(\underline{\omega}-\mu-1/\omega)).
        \label{eq:d-normalization}
    \end{align}
\end{thm}

To prove Theorem \ref{thm:DO-initial},
we use the following identity.

\begin{prop}\label{prop:saalschutz}
    For generic $(u_{1}, u_{2}, u_{4}, u_{5})\in\mathbb{C}^{4}$,
    it holds that
    \begin{align*}
         &
        \int_{C}du
        \exp{\left((4i\underline{\omega}-u_{1}-u_{2}-u_{4}-u_{5})\pi i \omega u\right)}
        \frac{G(u-u_{4})G(u-u_{5})}{G(u+u_{1})G(u+u_{2})}                                                  \\
         & =\frac{1}{\sqrt{\omega}}
        \exp{\left((u_{1}u_{2}-u_{4}u_{5}-i\underline{\omega}(u_{4}+u_{5}-u_{1}-u_{2}))\pi i\omega\right)} \\
         & \quad {}\times
        G(-3i\underline{\omega}+u_{1}+u_{2}+u_{4}+u_{5})
        \prod_{j=1}^{2}\prod_{k=4}^{5}G(i\underline{\omega}-u_{j}-u_{k}),
    \end{align*}
    where $C$ is a deformation of the real line $\mathbb{R}$
    separating the poles at
    $-u_{j}+i\underline{\omega}+i\mathbb{Z}_{\ge 0}+(i/\omega)\mathbb{Z}_{\ge 0}\,
        (j=1, 2)$ from
    the poles at
    $u_{k}-i\underline{\omega}-i\mathbb{Z}_{\ge 0}-(i/\omega)\mathbb{Z}_{\ge 0}\,
        (k=4, 5)$.
\end{prop}

\begin{proof}
    By using \eqref{eq:G-estimate0} and \eqref{eq:G'/G-estimate},
    we see that the both hand sides of the desired equality
    are holomorphic at generic $(u_{1}, u_{2}, u_{4}, u_{5})$.
    Therefore it is enough to prove it under the assumption
    $\mathop{\mathrm{Im}}{(u_{1}+u_{2}+u_{4}+u_{5})}>4\underline{\omega}$.

    We start from a hyperbolic analogue of non-terminating Saalsch\"utz formula
    \cite[Corollary 4.6]{BRS}:
    \begin{align}
        \int_{C_{0}}du \,
        \frac{\prod_{k=4}^{6}G(u-u_{k})}{\prod_{j=1}^{3}G(u+u_{j})}=
        \frac{1}{\sqrt{\omega}}
        \prod_{j=1}^{3}\prod_{k=4}^{6}
        G(i\underline{\omega}-u_{j}-u_{k}).
        \label{eq:GGG/GGG}
    \end{align}
    This equality holds for generic $(u_{1}, \ldots , u_{6})\in\mathbb{C}^{6}$ satisfying
    $\sum_{j=1}^{6}u_{j}=4i\underline{\omega}$.
    The contour $C_{0}$ is a deformation of the real line $\mathbb{R}$
    separating the poles at
    $-u_{j}+i\underline{\omega}+i\mathbb{Z}_{\ge 0}+(i/\omega)\mathbb{Z}_{\ge 0}\,
        (j=1, 2, 3)$ from
    the poles at
    $u_{k}-i\underline{\omega}-i\mathbb{Z}_{\ge 0}-(i/\omega)\mathbb{Z}_{\ge 0}\,
        (k=4, 5, 6)$.

    Take $u_{3}$ and $u_{6}$ satisfying
    $\sum_{j=1}^{6}u_{j}=4i\underline{\omega}, \,
        \mathop{\mathrm{Im}}{(u_{3}+u_{6})}<0$ and
    $\mathop{\mathrm{Im}}{(i\underline{\omega}-u_{j})}>0 \,
        (j=3, 6)$.
    Suppose that $\eta>0$ and set
    \begin{align*}
        \psi(u_{3}, u_{6}, \eta)=-\pi i \omega(u_{3}+u_{6})(\eta+\frac{u_{3}-u_{6}}{2}).
    \end{align*}
    From \eqref{eq:GGG/GGG} we see that
    \begin{align}
         &
        \int_{C_{0}}du \,
        \frac{\prod_{k=4,5}G(u-u_{k})}{\prod_{j=1, 2}G(u+u_{j})}\,
        e^{\psi(u_{3}, u_{6}, \eta)}
        \frac{G(u-u_{6}+\eta)}{G(u+u_{3}+\eta)}
        \label{eq:proof-GG/GG}      \\
         & =\frac{1}{\sqrt{\omega}}
        e^{\psi(u_{3}, u_{6}, \eta)}
        \prod_{j=1, 2}G(i\underline{\omega}-u_{j}-u_{6}+\eta)
        \prod_{k=4, 5}G(i\underline{\omega}-u_{3}-u_{k}-\eta)
        \nonumber                   \\
         & \quad {}\times
        G(-3i\underline{\omega}+u_{1}+u_{2}+u_{4}+u_{5})
        \prod_{j=1, 2}\prod_{k=4, 5}G(i\underline{\omega}-u_{j}-u_{k}).
        \nonumber
    \end{align}
    Note that, for large enough $\eta$,
    the contour $C_{0}$ is independent of $\eta$.
    Denote by $F(u)$ the integrand in the left hand side.
    Using \eqref{eq:G-estimate0} and
    $\mathop{\mathrm{Im}}(u_{3}+u_{6})<0$, we see that
    there exists a constant $\widetilde{M}$ which is independent of $\eta$ such that
    \begin{align*}
        |F(u)|\le \widetilde{M}
        \exp{(\mp 2\pi (1+\omega) \mathop{\mathrm{Re}}{u})} \qquad
        (u\in C_{0}, \, \mathop{\mathrm{Re}}{u}\gtrless A),
    \end{align*}
    where $A=\max{(\mathop{\mathrm{Re}}{u_{4}}, \mathop{\mathrm{Re}}{u_{5}},
            {}-\mathop{\mathrm{Re}}{u_{1}}, {}-\mathop{\mathrm{Re}}{u_{2}})}+
        \max{(1, 1/\omega)}$.
    Therefore, we can apply the dominated convergence theorem
    to the integral \eqref{eq:proof-GG/GG} and take the limit
    as $\eta \to +\infty$.
    Then we get the desired equality by using
    \eqref{eq:G-inverse} and \eqref{eq:G-limit}.
\end{proof}

\begin{proof}[Proof of Theorem \ref{thm:DO-initial}]
    By using \eqref{eq:G-property}, we see that
    \begin{align}
         &
        \mathcal{I}(\bk, \{1\}\,|\,\lambda, \mu)
        \label{eq:DO-initial-pf1}        \\
         & =
        (2\pi i \omega)^{|\bk|+1}
        \prod_{a=1}^{r}\int_{-\epsilon+i\,\mathbb{R}}
        \frac{dt_{a}}{e^{2\pi i t_{a}}-1}
        \prod_{a=1}^{r-1}J_{k_{a}}(t_{1}+\cdots +t_{a}\,|\,\lambda,\mu)
        \left(\frac{e^{2\pi i \omega [t]}}{1-e^{2\pi i \omega [t]}}\right)^{k_{r}-1}
        \nonumber                        \\
         & \qquad \qquad \qquad {}\times
        \frac{G(i(\underline{\omega}+[t]))}
        {G(i(\underline{\omega}+[t]+\lambda))G(i(\underline{\omega}+[t]+\mu))}
        e^{-\pi i \omega([t]+\lambda+\mu)}
        H([t]\,|\,\lambda, \mu),
        \nonumber
    \end{align}
    where
    \begin{align*}
        H(t\,|\,\lambda,\mu)=
        \int_{-\epsilon+i\,\mathbb{R}}du\,
        e^{-\pi i (1+\omega+\omega t)u}\,
        \frac{G(i(\underline{\omega}+u-1/\omega))G(i(\underline{\omega}+t+u+\lambda+\mu-1))}
        {G(i(\underline{\omega}+u+\lambda))G(i(\underline{\omega}+u+\mu))}.
    \end{align*}
    Proposition \ref{prop:saalschutz} implies that
    \begin{align*}
        H(t\,|\,\lambda,\mu)=d(\lambda, \mu)
        \frac{e^{\pi i \omega(2(\lambda+\mu)+3t)}}
        {(1-e^{2\pi i \omega (\lambda+t)})(1-e^{2\pi i \omega (\mu+t)})}\,
        \frac{G(i(\underline{\omega}+\lambda+t))G(i(\underline{\omega}+\mu+t))}
        {G(i(\underline{\omega}+t))}.
    \end{align*}
    Substituting it into \eqref{eq:DO-initial-pf1},
    we obtain the desired equality from
    Proposition \ref{prop:Ohno-function-integral}.
\end{proof}

\subsection{Transport relation}

We show an important property of
the connected integral $\mathcal{I}(\bk, \bl\,|\,\lambda,\mu)$.
For an index $\bk=(k_{1}, \ldots , k_{r})$, we set
\begin{align*}
    \bk_{\rightarrow}=(k_{1}, \ldots , k_{r}, 1),    \qquad
    \bk_{\rightarrow \uparrow}=(\bk_{\rightarrow})_{\uparrow}=(k_{1}, \ldots , k_{r}, 2).
\end{align*}

\begin{thm}[Transport relations]
    \label{thm:DO-transport}
    For any index $\bk$ and $\bl$, it holds that
    \begin{align}
        \mathcal{I}(\bk_{\rightarrow}, \bl\,|\,\lambda, \mu)
         & =
        \mathcal{I}(\bk, \bl_{\uparrow}\,|\,\lambda, \mu)+
        \frac{e^{2\pi i \omega \lambda}-1}{2\pi i \omega}
        \frac{e^{2\pi i \omega \mu}-1}{2\pi i \omega}
        \mathcal{I}(\bk_{\rightarrow\uparrow}, \bl_{\uparrow}\,|\,\lambda, \mu),
        \label{eq:transport1} \\
        \mathcal{I}(\bk_{\uparrow}, \bl\,|\,\lambda, \mu)
         & =
        \mathcal{I}(\bk, \bl_{\rightarrow}\,|\,\lambda, \mu)-
        \frac{e^{2\pi i \omega \lambda}-1}{2\pi i \omega}
        \frac{e^{2\pi i \omega \mu}-1}{2\pi i \omega}
        \mathcal{I}(\bk_{\uparrow}, \bl_{\rightarrow\uparrow}\,|\,\lambda, \mu)
        \label{eq:transport2}
    \end{align}
    in a neighborhood of $(\lambda, \mu)=(0, 0)$.
\end{thm}

\begin{proof}
    It is enough to prove \eqref{eq:transport1}
    since the symmetry
    $\mathcal{I}(\bk, \bl\,|\,\lambda,\mu)=\mathcal{I}(\bl,\bk\,|\,\lambda,\mu)$
    implies \eqref{eq:transport2}.

    Set $\bk=(k_{1}, \ldots , k_{r})$ and $\bl=(l_{1}, \ldots , l_{s})$.
    We see that
    \begin{align}
         &
        \mathcal{I}(\bk_{\rightarrow}, \bl\,|\,\lambda, \mu)-
        \frac{e^{2\pi i \omega \lambda}-1}{2\pi i \omega}
        \frac{e^{2\pi i \omega \mu}-1}{2\pi i \omega}
        \mathcal{I}(\bk_{\rightarrow\uparrow}, \bl_{\uparrow}\,|\,\lambda, \mu)
        \label{eq:transport-proof}          \\
         & =(2\pi i \omega)^{|\bk|+|\bl|+1}
        \prod_{a=1}^{r}\int_{-\epsilon+i\,\mathbb{R}}\frac{dt_{a}}{e^{2\pi i t_{a}}-1}
        \prod_{b=1}^{s}\int_{-\epsilon+i\,\mathbb{R}}\frac{du_{b}}{e^{2\pi i u_{b}}-1}
        \nonumber                           \\
         & \quad {}\times
        \prod_{a=1}^{r}J_{k_{a}}(t_{1}+\cdots +t_{a}\,|\,\lambda, \mu)
        \prod_{b=1}^{s}J_{l_{b}}(u_{1}+\cdots +u_{b}\,|\,\lambda, \mu)
        \left(\frac{e^{2\pi i \omega [u]}}{1-e^{2\pi i \omega [u]}}\right)^{l_{s}-1}
        \nonumber                           \\
         & \quad {}\times
        \frac{e^{-\pi i \omega [t][u]}}{1-e^{2\pi i \omega [u]}}
        \frac{G(i(\underline{\omega}+[u]))}
        {G(i(\underline{\omega}+[u]+\lambda))G(i(\underline{\omega}+[u]+\mu))}
        V([t], [u], \lambda, \mu),
        \nonumber
    \end{align}
    where $[t]=\sum_{a=1}^{r}t_{a}, [u]=\sum_{b=1}^{s}u_{b}$ and
    \begin{align*}
        V(t, u, \lambda, \mu)
         & =\int_{-\epsilon+i\,\mathbb{R}} \frac{ds}{e^{2\pi i s}-1}
        \left\{ 1-e^{2\pi i \omega u}
        \frac{(1-e^{2\pi i \omega (t+\lambda+s)})(1-e^{2\pi i \omega (t+\mu+s)})}
        {(1-e^{2\pi i \omega (t+s)})(1-e^{2\pi i \omega (t+u+\lambda+\mu+s)})}
        \right\}                                                     \\
         & \qquad {}\times
        \Xi(s\,|\,t, u, \lambda, \mu),                               \\
        \Xi(s\,|\,t, u, \lambda, \mu)
         & =
        e^{-\pi i \omega us}
        \frac{G(i(\underline{\omega}+t+s)) G(i(\underline{\omega}+t+u+\lambda+\mu+s))}
        {G(i(\underline{\omega}+t+\lambda+s)) G(i(\underline{\omega}+t+\mu+s))}.
    \end{align*}
    Using \eqref{eq:G-property}, we see that
    \begin{align*}
        V(t, u, \lambda, \mu)
         & =
        \int_{-\epsilon+i\,\mathbb{R}}
        \frac{ds}{e^{2\pi i s}-1}
        \left(\Xi(s\,|\,t,u,\lambda,\mu)-\Xi(s-1\,|\,t, u, \lambda, \mu)\right) \\
         & =\left(
        \int_{-\epsilon+i\,\mathbb{R}}-\int_{-\epsilon-1+i\,\mathbb{R}}
        \right)
        \frac{ds}{e^{2\pi i s}-1}\Xi(s\,|\,t,u,\lambda,\mu)                     \\
         & =2\pi i \mathop{\mathrm{Res}}_{s=-1}\frac{ds}{e^{2\pi i s}-1}
        \Xi(s\,|\,t, u, \lambda, \mu)                                           \\
         & =e^{2\pi i \omega u}
        \frac{(1-e^{2\pi i \omega (t+\lambda)})(1-e^{2\pi i \omega (t+\mu)})}
        {(1-e^{2\pi i \omega t})(1-e^{2\pi i \omega (t+u+\lambda+\mu)})}
        \frac{G(i(\underline{\omega}+t))G(i(\underline{\omega}+t+u+\lambda+\mu))}
        {G(i(\underline{\omega}+t+\lambda))G(i(\underline{\omega}+t+\mu))}.
    \end{align*}
    {}From this formula we see that
    the right hand side of \eqref{eq:transport-proof} is equal to
    $\mathcal{I}(\bk, \bl_{\uparrow}\,|\,\lambda, \mu)$.
\end{proof}

\subsection{Proof of extended double Ohno relations}

Now we can prove Theorem \ref{thm:extended-double-Ohno} in a similar way to
the proof given in \cite{HSS}.

Recall the definition \eqref{eq:d-normalization}
of $d(\lambda, \mu)$.
We see that $1/d(\lambda, \mu)$ is holomorphic
in a neighborhood of $(\lambda, \mu)=(0, 0)$.
Hence we can define the value $Z_{m, n}(\bk, \bl)$ for
indices $\bk, \bl$ and $m, n\ge 0$ by the expansion
\begin{align*}
    \frac{\mathcal{I}(\bk, \bl \,|\, \lambda, \mu)}{d(\lambda, \mu)}=
    \sum_{m, n\ge 0}Z_{m, n}(\bk, \bl)
    \left(\frac{e^{2\pi i \omega \lambda}-1}{2\pi i \omega}\right)^{m}
    \left(\frac{e^{2\pi i \omega \mu}-1}{2\pi i \omega}\right)^{n}.
\end{align*}
We define the symmetric $\mathbb{Q}$-bilinear map
$\Phi(\cdot, \cdot\,|\,\xi,\eta): y\mathfrak{h}\times y\mathfrak{h} \to
    \mathbb{C}[[\xi, \eta]]$ by
\begin{align*}
    \Phi(z_{\bk}, z_{\bl}\,|\, \xi, \eta)=
    \sum_{m, n\ge 0}Z_{m, n}(\bk, \bl)\xi^{m}\eta^{n}
\end{align*}
for any index $\bk$ and $\bl$,
and extend it to the $\mathbb{Q}$-bilinear map
$\Phi(\cdot, \cdot\,|\,\xi,\eta): y\mathfrak{h}[[X]]\times y\mathfrak{h}[[X]] \to
    \mathbb{C}[[\xi, \eta]]$ by
\begin{align*}
    \Phi({\textstyle \sum_{i\ge 0}w_{i}X^{i}},
    {\textstyle \sum_{j\ge 0}w_{j}'X^{j}} \,|\,\xi,\eta)=
    \sum_{i,j\ge 0}\Phi(w_{i}, w_{j}'\,|\,\xi,\eta)(\xi\eta)^{i+j},
\end{align*}
where $w_{i}, w_{j}' \in y\mathfrak{h}$.

Theorem \ref{thm:DO-transport} implies that
\begin{align}
    \nonumber
    \Phi(wy, w'\,|\,\xi,\eta) & =
    \Phi(w, w'x\,|\,\xi,\eta)+\xi\eta\,\Phi(wyx, w'x\,|\,\xi,\eta),
    \\
    \label{eq:transport-phi-x}
    \Phi(wx, w'\,|\,\xi,\eta) & =
    \Phi(w, w'y\,|\,\xi,\eta)-\xi\eta\,\Phi(wx, w'yx\,|\,\xi,\eta)
\end{align}
for $w, w' \in y\mathfrak{h}[[X]]$.
These relations imply that
$\Phi(wyx, w'\,|\,\xi,\eta)=\Phi(w, w'yx\,|\,\xi,\eta)$.
Hence it holds that
\begin{align*}
    \Phi(wy, w'\,|\,\xi,\eta)
     & =
    \Phi(w, w'x\,|\,\xi,\eta)+\xi\eta\,\Phi(w, w'xyx\,|\,\xi,\eta) \\
     & =
    \Phi(w, w'x(1+yx X)\,|\,\xi,\eta)=
    \Phi(w, w'\tau(y)\,|\,\xi,\eta).
\end{align*}
Using \eqref{eq:transport-phi-x} repeatedly,
we see that
\begin{align*}
    \Phi(wx, w'\,|\,\xi,\eta)=
    \sum_{j=0}^{N}(-\xi\eta)^{j}\Phi(w, w'y(xy)^{j}\,|\,\xi,\eta)+
    (-\xi\eta)^{N+1}\Phi(wx, w'(yx)^{N+1}\,|\,\xi,\eta)
\end{align*}
for all $N\ge 1$.
Therefore, in the formal power series ring $\mathbb{C}[[\xi, \eta]]$,
it holds that
\begin{align*}
    \Phi(wx, w'\,|\,\xi,\eta)
     & =
    \sum_{j=0}^{\infty}(-\xi\eta)^{j}\Phi(w, w'y(xy)^{j}\,|\,\xi,\eta) \\
     & =
    \Phi(w, w'y(1+xyX)^{-1}\,|\,\xi,\eta)
    =\Phi(w, w'\tau(x)\,|\,\xi,\eta).
\end{align*}
Thus we obtain
\begin{align}
    \Phi(wv, w'\,|\,\xi,\eta)=\Phi(w, w'\tau(v)\,|\,\xi,\eta)
    \label{eq:DO-pf1}
\end{align}
for $w, w' \in y\mathfrak{h}[[X]]$ and $v\in \{x, y\}$.

Theorem \ref{thm:DO-initial} and
the symmetry of $\Phi$ show that
\begin{align*}
    \Omega(ywx \,|\,\xi,\eta)=\Phi(yw, y \,|\, \xi,\eta), \qquad
    \Omega(y\tau(w)x \,|\,\xi,\eta)=\Phi(y, y\tau(w)\,|\,\xi,\eta)
\end{align*}
for $w \in \mathfrak{h}$.
Using \eqref{eq:DO-pf1} repeatedly, we see that
$\Phi(yw, y \,|\, \xi,\eta)=\Phi(y, y\tau(w)\,|\,\xi,\eta)$.
Thus we obtain the desired equality
$\Omega(ywx \,|\,\xi,\eta)=\Omega(y\tau(w)x \,|\,\xi,\eta)$.



\appendix

\section{Proof of Proposition \ref{prop:omega-MZV-limit1}
  and Proposition \ref{prop:limit-Zg}}
\label{sec:app1}

We fix two constants $\sigma$ and $A$ satisfying $0<\sigma<1$ and $0<A<\pi$.
Then there exists a positive constant
$M=M(\sigma, A)$ such that
\begin{align*}
    \left|\frac{z}{\sinh{z}}\right|\le Me^{-\sigma |\mathop{\mathrm{Re}}{z}|}
\end{align*}
if $|\mathop{\mathrm{Im}}{z}|\le A$.
Hence, if $-A/\pi \omega<\mathop{\mathrm{Re}}{t}<0$, it holds that
\begin{align*}
    \left|\frac{2\pi i \omega e^{2\pi i \omega t}}{1-e^{2\pi i \omega t}}\right|=
    \left|\frac{\pi i \omega t}{\sinh{(\pi i \omega t)}}\frac{e^{\pi i \omega t}}{t}\right|
    \le
    M \frac{e^{-\pi \omega (\sigma|\mathop{\mathrm{Im}}{t}|+\mathop{\mathrm{Im}}{t})}}{|t|}.
\end{align*}

For $\alpha\ge 1$, there exists a positive constant $B_{\alpha}$
such that
\begin{align*}
    \left|\binom{t+\alpha}{\alpha}\right|\le
    \binom{|t|+\alpha}{\alpha}\le B_{\alpha}(|t|+1)^{\alpha}
\end{align*}
for any $t \in \mathbb{C}$.

Now assume that $0<\epsilon<\min{(1, 1/r\omega, A/\pi r \omega)}$.
Denote by $E(t_{1}, \ldots , t_{r})$
the integrand of the right hand side of \eqref{eq:omegaMZV-formula}.
The above inequalities imply that
\begin{align}
    \left|E(t_{1}, \ldots , t_{r})\right|
     & \le C_{1} \omega^{\sum_{a=1}^{r}\alpha_{a}}
    \prod_{a=1}^{r}\frac{1}{|e^{2\pi i t_{a}}-1|}
    \frac{(|t_{a}|+1)^{\alpha_{a}}}{|t_{1}+\cdots +t_{a}|^{\beta_{a}+1}}
    \label{eq:omega-MZV-estimate1}                 \\
     & \times
    \exp{(
    -\pi\omega \sum_{a=1}^{r}(\beta_{a}+1)
    (\sigma|\mathop{\mathrm{Im}}{(t_{1}+\cdots +t_{a})}|+
    \mathop{\mathrm{Im}}{(t_{1}+\cdots +t_{a})})
    )}
    \nonumber
\end{align}
on the integral region, where $C_{1}$ is a constant independent of $\omega$.
Now set
$t_{a}=-\epsilon+i\epsilon x_{a}$
for $1\le a \le r$.
It holds that
\begin{align*}
    |t_{a}|+1\le \epsilon(1+|x_{a}|)+1\le (\epsilon+1)(|x_{a}|+1)
\end{align*}
and
\begin{align*}
    |t_{1}+\cdots +t_{a}|=\epsilon\sqrt{a^{2}+(x_{1}+\cdots+x_{a})^{2}}
    \ge \frac{\epsilon}{\sqrt{2}}
    (1+|x_{1}+\cdots +x_{a}|)
\end{align*}
for $1\le a \le r$.
Hence, using Proposition \ref{prop:I-estimate},
we see that the right hand side of \eqref{eq:omega-MZV-estimate1} is
estimated from above by
\begin{align*}
    C_{2} \, \omega^{\sum_{a=1}^{r}\alpha_{a}}
    \prod_{a=1}^{r}\frac{(|x_{a}|+1)^{\alpha_{a}}}{(|x_{1}+\cdots +x_{a}|+1)^{\beta_{a}+1}}
    \exp{(\pi \epsilon S(x_{1}, \ldots , x_{r}))},
\end{align*}
where $C_{2}$ is a constant independent of $\omega$ and
\begin{align*}
    S(x_{1}, \ldots , x_{r})=
    \sum_{a=1}^{r}(x_{a}-|x_{a}|)-\omega
    \sum_{a=1}^{r}(\beta_{a}+1)(\sigma |x_{1}+\cdots +x_{a}|+x_{1}+\cdots +x_{a}).
\end{align*}
It holds that
\begin{align*}
     &
    S(x_{1}, \ldots , x_{r})                                 \\
     & = \frac{1}{2}\sum_{a=1}^{r}(x_{a}-|x_{a}|)+
    \frac{1}{2}\sum_{a=1}^{r}
    \left\{(1-2\omega\sum_{j=a}^{r}(\beta_{j}+1))x_{a}-|x_{a}|\right\}-\omega
    \sigma \sum_{a=1}^{r} (\beta_{a}+1)|x_{1}+\cdots +x_{a}| \\
     & \le
    \frac{1}{2}\sum_{a=1}^{r}(x_{a}-|x_{a}|)+
    \frac{1}{2}\sum_{a=1}^{r}
    (\left|1-2\omega\sum_{j=a}^{r}(\beta_{j}+1)\right|-1)|x_{a}|-\omega
    \sigma \sum_{a=1}^{r} |x_{1}+\cdots +x_{a}|.
\end{align*}
Now assume that
\begin{align*}
    0<\omega<(2\sum_{j=1}^{r}(\beta_{j}+1))^{-1}.
\end{align*}
Then it holds that
\begin{align*}
    S(x_{1}, \ldots , x_{r})
     & \le
    \frac{1}{2}\sum_{a=1}^{r}(x_{a}-|x_{a}|)-\omega
    \sum_{a=1}^{r} \sum_{j=a}^{r}(\beta_{j}+1)|x_{a}|-\omega \sigma
    \sum_{a=1}^{r} |x_{1}+\cdots +x_{a}| \\
     & \le
    \frac{1}{2}\sum_{a=1}^{r}(x_{a}-|x_{a}|)-\omega\sigma
    \sum_{a=1}^{r} |x_{1}+\cdots +x_{a}|.
\end{align*}
As a result we obtain the estimate
\begin{align}
    |E(t_{1}, \ldots , t_{r})|
     & \le
    C_{2} \,
    \omega^{\sum_{a=1}^{r}\alpha_{a}}
    \prod_{a=1}^{r}\frac{(|x_{a}|+1)^{\alpha_{a}}}{(|x_{1}+\cdots +x_{a}|+1)^{\beta_{a}+1}}
    \label{eq:omega-MZV-integrand-estimate} \\
     & \times
    \exp{(\pi \epsilon \left(\frac{1}{2}\sum_{a=1}^{r}(x_{a}-|x_{a}|)-\omega \sigma
        \sum_{a=1}^{r}|x_{1}+\cdots +x_{a}| \right))}
    \nonumber
\end{align}
for sufficiently small $\omega>0$,
where $t_{a}=-\epsilon+i\epsilon x_{a} \, (1\le a \le r)$.

\begin{proof}[Proof of Proposition \ref{prop:omega-MZV-limit1} (i)]

    {}From \eqref{eq:omega-MZV-integrand-estimate} we see that
    \begin{align*}
         &
        \left|
        Z_{\omega}((e_{1}-g_{1})^{\alpha_{1}}g_{\beta_{1}+1}\cdots
        (e_{1}-g_{1})^{\alpha_{r}}g_{\beta_{r}+1})
        \right| \\
         & \le
        C_{3}\, \omega^{\sum_{a=1}^{r}\alpha_{a}}
        \left(\prod_{a=1}^{r}\int_{\mathbb{R}}dx_{a}\right)
        \prod_{a=1}^{r}\frac{(|x_{a}|+1)^{\alpha_{a}}}{|x_{1}+\cdots +x_{a}|+1}
        e^{-\pi \epsilon \sigma \omega \sum_{a=1}^{r}|x_{1}+\cdots +x_{a}|},
    \end{align*}
    for sufficiently small $\omega>0$,
    where $C_{3}$ is a constant independent of $\omega$.
    Now we change the integration variable to $y_{a}=x_{1}+\cdots +x_{a}$
    for $1\le a \le r$.
    We have
    \begin{align*}
        \prod_{a=1}^{r}(|x_{a}|+1)^{\alpha_{a}}
         & =
        (|y_{1}|+1)^{\alpha_{1}}
        \prod_{a=2}^{r}(|y_{a}-y_{a-1}|+1)^{\alpha_{a}}          \\
         & \le
        (|y_{1}|+1)^{\alpha_{1}}
        \prod_{a=2}^{r}((|y_{a}|+1)+(|y_{a-1}|+1))^{\alpha_{a}}  \\
         & =\sum_{\substack{\gamma_{1}, \ldots , \gamma_{r}\ge 0 \\
                \gamma_{1}+\cdots +\gamma_{r}=\alpha_{1}+\cdots +\alpha_{r}}}
        d_{\gamma_{1}, \ldots , \gamma_{r}}^{\alpha_{1}, \ldots , \alpha_{r}}
        \prod_{a=1}^{r}(|y_{a}|+1)^{\gamma_{a}},
    \end{align*}
    where $d_{\gamma_{1}, \ldots , \gamma_{r}}^{\alpha_{1}, \ldots , \alpha_{r}}$
    is a non-negative integer.
    Hence
    \begin{align*}
         &
        \left|
        Z_{\omega}((e_{1}-g_{1})^{\alpha_{1}}g_{\beta_{1}+1}\cdots
        (e_{1}-g_{1})^{\alpha_{r}}g_{\beta_{r}+1})
        \right|                                              \\
         & \le C_{3}
        \sum_{\substack{\gamma_{1}, \ldots , \gamma_{r}\ge 0 \\
                \gamma_{1}+\cdots +\gamma_{r}=\alpha_{1}+\cdots +\alpha_{r}}}
        d_{\gamma_{1}, \ldots , \gamma_{r}}^{\alpha_{1}, \ldots , \alpha_{r}}
        \prod_{a=1}^{r}\left(
        \omega^{\gamma_{a}} \int_{\mathbb{R}}dy\,
        (|y|+1)^{\gamma_{a}-1}e^{-\pi \epsilon \sigma \omega |y|}
        \right).
    \end{align*}
    By using Lemma \ref{lem:F-gamma} below, we see that the right hand side above is
    $O((-\log{\omega})^{r})$ as $\omega \to +0$.
\end{proof}

\begin{lem}\label{lem:F-gamma}
    For a non-negative integer $\gamma$, we set
    \begin{align*}
        F_{\gamma}(\rho)=\rho^{\gamma} \int_{0}^{\infty}dy\,(y+1)^{\gamma-1}e^{-\rho y }
        \qquad (\rho>0).
    \end{align*}
    Then it holds that $F_{0}(\rho)=O(-\log{\rho})$ and
    $F_{\gamma}(\rho)=O(1)$ for $\gamma \ge 1$ as $\rho\to +0$.
\end{lem}

\begin{proof}
    Since $\rho>0$, it holds that
    \begin{align*}
        F_{\gamma}(\rho)=\int_{0}^{\infty}dy (y+\rho)^{\gamma-1}e^{-y}.
    \end{align*}
    Hence
    \begin{align*}
        0\le F_{0}(\rho)\le
        \int_{0}^{1}\frac{dy}{y+\rho}+\int_{1}^{\infty}dy\,\frac{e^{-y}}{y}=
        \log{(1+\frac{1}{\rho})}+\int_{1}^{\infty}dy\,\frac{e^{-y}}{y}=O(-\log{\rho})
    \end{align*}
    and
    \begin{align*}
        0\le F_{\gamma}(\rho)\le \int_{0}^{\infty}(y+1)^{\gamma-1}e^{-y}<+\infty
    \end{align*}
    for $\gamma\ge 1$ and $\rho \le 1$.
\end{proof}

Next we prove Proposition \ref{prop:omega-MZV-limit1} (ii).
We use the following inequality.

\begin{lem}\label{eq:estimate-psi}
    Suppose that $\eta>0$ and $1\le s\le t \le r$.
    Then the function
    \begin{align*}
        \Psi(x_{1}, \ldots , x_{r})=e^{\eta\sum_{j=1}^{r}(x_{j}-|x_{j}|)}
        \frac{|x_{s}|+1}{|x_{1}+\cdots +x_{t}|+1}.
    \end{align*}
    satisfies
    \begin{align*}
        0 \le \Psi(x_{1}, \ldots , x_{r})\le
        (\max{(1,  1/2\eta)})^{2}
    \end{align*}
    for all $x_{1}, \ldots , x_{r} \in \mathbb{R}$.
\end{lem}

\begin{proof}
    Set $x=x_{s}$ and $y=\sum_{1\le j \le t, j\not=s}x_{j}$.
    Since $\eta(x_{j}-|x_{j}|)\le 0$ for $t<j\le r$ and
    $|x|+|y|\le \sum_{1\le j\le t}|x_{j}|$, we see that
    \begin{align*}
        0\le \Psi(x_{1}, \ldots , x_{r})\le
        e^{\eta\sum_{j=1}^{t}(x_{j}-|x_{j}|)}\frac{|x|+1}{|x+y|+1} \le
        e^{\eta(x-|x|+y-|y|)}\frac{|x|+1}{|x+y|+1}.
    \end{align*}
    Set $C=\max{(1, 1/2\eta)}$ and
    \begin{align*}
        \psi(x)=\left\{
        \begin{array}{cl}
            1 & (x\ge 0), \\ (1+|x|)^{-1} & (x\le 0).
        \end{array}
        \right.
    \end{align*}
    Using $0\le 1+t\le e^{t}$ for $t\ge 0$, we see that
    \begin{align*}
        e^{\eta(x-|x|)}\le C\psi(x)
    \end{align*}
    for any $x\in \mathbb{R}$.
    Therefore, it holds that
    \begin{align*}
        0\le \Psi(x_{1}, \ldots , x_{r})\le C^{2} \psi(x)\psi(y) \frac{|x|+1}{|x+y|+1}.
    \end{align*}
    We see that
    \begin{align*}
        0\le \psi(x)\psi(y) \frac{|x|+1}{|x+y|+1}\le 1
    \end{align*}
    by case checking, and this completes the proof.
\end{proof}

\begin{proof}[Proof of Proposition \ref{prop:omega-MZV-limit1} (ii)]
    Suppose that $\alpha_{s}, \beta_{t}\ge 1$ with $1\le s\le t \le r$.
        {}From \eqref{eq:omega-MZV-integrand-estimate}
    and Lemma \ref{eq:estimate-psi} with $\eta=\pi\epsilon/2$,
    we see that
    \begin{align*}
         &
        \left|
        Z_{\omega}((e_{1}-g_{1})^{\alpha_{1}}g_{\beta_{1}+1}\cdots
        (e_{1}-g_{1})^{\alpha_{r}}g_{\beta_{r}+1})
        \right|              \\
         & \le C_{4}
        \omega^{\sum_{a=1}^{r}\alpha_{a}}
        \left(\prod_{a=1}^{r}\int_{\mathbb{R}}dx_{a}\right)
        \frac{(|x_{s}|+1)^{\alpha_{s}-1}
        \prod_{\substack{a=1 \\ a\not=s}}^{r}(|x_{a}|+1)^{\alpha_{a}}}
        {\prod_{a=1}^{r}(|x_{1}+\cdots +x_{a}|+1)}
        e^{-\pi\epsilon\sigma\omega \sum_{a=1}^{r}|x_{1}+\cdots +x_{a}|},
    \end{align*}
    where $C_{4}$ is a constant independent of $\omega$.
    As shown in the proof of Proposition \ref{prop:omega-MZV-limit1} (i),
    the right hand side is $O(\omega (-\log{\omega})^{r})$,
    which goes to zero in the limit as $\omega \to +0$.
\end{proof}

We show the two lemmas below
in order to prove  Proposition \ref{prop:limit-Zg}.

\begin{lem}\label{lem:estimate-Im}
    Suppose that $\eta>0$.
    For a non-negative integer $m$, we set
    \begin{align*}
        I_{m}(s)=\int_{\mathbb{R}}dy\,
        \frac{(\log{(|y|+1)})^{m}}{|y|+1}\,e^{-\eta(|y|+|s-y|)}.
    \end{align*}
    Then there exists a positive constant $B=B(m, \eta)$ such that
    \begin{align*}
        0\le I_{m}(s) \le
        e^{-\eta|s|}\left(
        \frac{(\log{(|s|+1)})^{m+1}}{m+1}+B
        \right)
    \end{align*}
    for any $s \in \mathbb{R}$.
\end{lem}

\begin{proof}
    We see that $I_{m}(-s)=I_{m}(s)\ge 0$ from the definition.
    Hence it is enough to consider the case where $s\ge 0$.

    For simplicity we set
    \begin{align*}
        f_{m}(y)=\frac{(\log{(|y|+1)})^{m}}{|y|+1}.
    \end{align*}
    Note that $f_{m}(y)=f_{m}(-y)$ and
    $C_{m}=\sup_{y\in\mathbb{R}}|f_{m}(y)|$ is finite.
    We have
    \begin{align*}
        \int_{0}^{s}dy\,f_{m}(y)=\frac{(\log{(s+1)})^{m+1}}{m+1}
    \end{align*}
    for $s\ge 0$.
    Hence it holds that
    \begin{align*}
        I_{m}(s)
         & =
        \left(\int_{-\infty}^{0}+\int_{0}^{s}+\int_{s}^{\infty}\right)dy\,
        f_{m}(y)e^{-\eta(|y|+|s-y|)}                               \\
         & =e^{-\eta s}\int_{0}^{\infty}dy\,f_{m}(-y)e^{-2\eta y}+
        e^{-\eta s}\int_{0}^{s}dy\,f_{m}(y)+
        e^{-\eta s}\int_{0}^{\infty}dy\,f_{m}(y+s)e^{-2\eta y}     \\
         & \le
        e^{-\eta s}\frac{C_{m}}{\eta}+
        e^{-\eta s}\frac{(\log{(s+1)})^{m+1}}{m+1}.
    \end{align*}
    This completes the proof.
\end{proof}

\begin{lem}\label{lem:K-integrable}
    Suppose that $r$ and $k$ are positive integers, $\eta>0$ and $A\ge 1$.
    Then the function
    \begin{align}
        K_{\eta, A}(x_{1}, \ldots , x_{r})=
        \left(\prod_{a=1}^{r-1}\frac{1}{|x_{1}+\cdots +x_{a}|+1} \right)
        \frac{e^{\eta\sum_{a=1}^{r}(x_{a}-|x_{a}|)}}{(|x_{1}+\cdots +x_{r}|+A)^{k+1}}
        \label{eq:def-K-estimate}
    \end{align}
    is integrable on $\mathbb{R}^{r}$ and it holds that
    \begin{align*}
        \int_{\mathbb{R}^{r}}
        K_{\eta, A}(x_{1}, \ldots , x_{r})dx_{1}\cdots dx_{r}=
        O(A^{-k}(\log{A})^{r-1}) \qquad (A \to +\infty).
    \end{align*}
\end{lem}

\begin{proof}
    Change the integration variables to $y_{a}=x_{1}+\cdots +x_{a}$ for $1\le a \le r$,
    and we see that
    \begin{align*}
        \int_{\mathbb{R}^{r}}K_{\eta, A}(x_{1}, \ldots , x_{r})dx_{1}\cdots dx_{r}=
        \int_{\mathbb{R}}\frac{dy_{r}}{(|y_{r}|+A)^{k+1}}e^{\eta y_{r}}
        \left(
        \prod_{a=1}^{r-1}\int_{\mathbb{R}}\frac{dy_{a}}{|y_{a}|+1}
        e^{-\eta\sum_{a=1}^{r}|y_{a}-y_{a-1}|}
        \right),
    \end{align*}
    where $y_{0}=0$.
    Using Lemma \ref{lem:estimate-Im} repeatedly, we see that
    there exists a polynomial $P_{r-1}(x)$ of degree $(r-1)$
    such that
    \begin{align*}
        0\le
        \prod_{a=1}^{r-1}\int_{\mathbb{R}}\frac{dy_{a}}{|y_{a}|+1}
        e^{-\eta\sum_{a=1}^{r}|y_{a}-y_{a-1}|}
        \le
        e^{-\eta|y_{r}|} P_{r-1}(\log{(|y_{r}|+1)})
    \end{align*}
    for any $y_{r} \in \mathbb{R}$.
    Thus we see that
    \begin{align*}
        0
         & \le
        \int_{\mathbb{R}^{r}}K_{\eta, A}(x_{1}, \ldots , x_{r})dx_{1}\cdots dx_{r} \\
         & \le
        \int_{\mathbb{R}}dy
        \frac{P_{r-1}(\log{(|y|+1)})}{(|y|+A)^{k+1}}e^{\eta(y-|y|)}
        \le
        2 \int_{0}^{\infty}dy
        \frac{P_{r-1}(\log{(y+1)})}{(y+A)^{k+1}}<+\infty
    \end{align*}
    since $k\ge 1$.
    Let $m$ be a positive integer.
    Setting $y=At$, we see that
    \begin{align*}
        0
         & \le
        \int_{0}^{\infty}dy \, \frac{(\log{(y+1)})^{m}}{(y+A)^{k+1}}=
        A^{-k} \int_{0}^{\infty} dt \, \frac{(\log{A}+\log{(1/A+t)})^{m}}{(t+1)^{k+1}} \\
         & \le
        A^{-k} \int_{0}^{\infty} dt \, \frac{(\log{A}+\log{(1+t)})^{m}}{(t+1)^{k+1}}=
        O(A^{-k}(\log{A})^{m})
    \end{align*}
    as $A \to +\infty$.
    Since the degree of $P_{r-1}(x)$ is equal to $r-1$, we obtain the desired result.
\end{proof}

\begin{proof}[Proof of Proposition \ref{prop:limit-Zg}]
    Suppose that $\bk=(k_{1}, \ldots , k_{r})$ is admissible.
        {}From the definition we have
    \begin{align}
        Z_{\omega}(g_{k_{1}}\cdots g_{k_{r}})=
        \prod_{a=1}^{r}\int_{-\epsilon+i\,\mathbb{R}}\frac{dt_{a}}{e^{2\pi i t_{a}}-1}
        \prod_{a=1}^{r}\left(
        \frac{2\pi i \omega}{e^{-2\pi i \omega (t_{1}+\cdots +t_{a})}-1}
        \right)^{k_{a}}.
        \label{eq:omega-MZV-g-integral}
    \end{align}
    We denote by $E(t_{1}, \ldots , t_{r})$ the above integrand.
    Set $t_{a}=-\epsilon+i\epsilon x_{a}$
    for $1\le a \le r$ and $\eta=\pi\epsilon/2$.
    Then we have \eqref{eq:omega-MZV-integrand-estimate}
    with $\alpha_{a}=0$ and $\beta_{a}+1=k_{a}$ for $1\le a \le r$.
    Furthermore, since $k_{r}\ge 2$, it holds that
    $|E(t_{1}, \ldots , t_{r})|\le C_{2}K(x_{1}, \ldots , x_{r})$,
    where $K(x_{1}, \ldots , x_{r})$ is the function \eqref{eq:def-K-estimate}
    with $A=1$ and $k=1$,
    and it is integrable because of Lemma \ref{lem:K-integrable}.
    Hence we can apply the dominated convergence theorem in the limit as $\omega\to+0$,
    and we see that
    \begin{align*}
        \lim_{\omega \to +0}
        Z_{\omega}(g_{k_{1}} \cdots g_{k_{r}})=(-1)^{|\bk|}
        \prod_{a=1}^{r}\int_{-\epsilon+i\,\mathbb{R}}\frac{dt_{a}}{e^{2\pi i t_{a}}-1}
        \prod_{a=1}^{r}\frac{1}{(t_{1}+\cdots +t_{a})^{k_{a}}}.
    \end{align*}
    Let $M$ be a positive integer and
    shift the integral contour of $t_{r}, t_{r-1}, \ldots , t_{1}$
    to $-\epsilon-M+i\,\mathbb{R}$ in order
    by taking the residues at $t_{a}=-l_{a} \in \{-1, -2, \ldots, -M\}$.
    Then the right hand side above becomes
    \begin{align}
        \sum_{l_{1}, \ldots , l_{r}=1}^{M}
        \prod_{a=1}^{r}\frac{1}{(l_{1}+\cdots +l_{a})^{k_{a}}}+(-1)^{|\bk|}
        \sum_{b=1}^{r}\sum_{l_{b+1}, \ldots , l_{r}=1}^{M}
        S_{b}(l_{b+1}, \ldots , l_{r}\,|\, M),
        \label{eq:limit-Zg-pf1}
    \end{align}
    where
    \begin{align*}
        S_{b}(l_{b+1}, \ldots , l_{r}\,|\, M)
         & =
        \left(\prod_{a=1}^{b-1}\int_{-\epsilon+i\,\mathbb{R}}
        \frac{dt_{a}}{e^{2\pi i t_{a}}-1}
        \right)
        \int_{-\epsilon-M+i\,\mathbb{R}}
        \frac{dt_{b}}{e^{2\pi i t_{b}}-1} \\
         & \times
        \prod_{a=1}^{b}\frac{1}{(t_{1}+\cdots +t_{a})^{k_{a}}}
        \prod_{a=b+1}^{r}\frac{1}{(t_{1}+\cdots +t_{b}-l_{b+1}-\cdots -l_{r})^{k_{a}}}.
    \end{align*}
    Since the first term of \eqref{eq:limit-Zg-pf1} converges to
    the MZV $\zeta(k_{1}, \ldots , k_{r})$ as $M\to \infty$,
    it suffices to show that the second term vanishes
    in the limit.
    Set $t_{a}=-\epsilon+i\epsilon x_{a}$ for $1\le a \le b-1$
    and $t_{b}=-\epsilon-M+i\epsilon x_{b}$.
    Then, using Lemma \ref{lem:estimate-below}, we see that
    \begin{align*}
         &
        \left|S_{b}(l_{b+1}, \ldots , l_{r}\,|\,M)\right| \\
         & \le
        C_{5} \prod_{a=1}^{b}\int_{\mathbb{R}}dx_{a} \,
        \left(\prod_{a=1}^{b-1}\frac{1}{|x_{1}+\cdots +x_{a}|+1}\right)
        \frac{e^{\pi \epsilon \sum_{a=1}^{b}(x_{a}-|x_{a}|)}}
        {(|x_{1}+\cdots +x_{b}|+A)^{k_{b}+k_{b+1}+\cdots +k_{r}}},
    \end{align*}
    where $C_{5}$ is a constant independent of $M$ and
    $A=b+M/\epsilon$.
    Note that the right hand side above does not depend on
    $l_{b+1}, \ldots , l_{r}$ and
    behaves as $O(M^{-(k_{b}+\cdots +k_{r}-1)}(\log{M})^{b-1})$
    as $M\to \infty$ because of Lemma \ref{lem:K-integrable} and
    $k_{r}\ge 2$.
    Therefore, we have
    \begin{align*}
        \sum_{l_{b+1}, \ldots , l_{r}=1}^{M}
        S_{b}(l_{b+1}, \ldots , l_{r}\,|\, M)
         & =
        M^{r-b}O(M^{-(k_{b}+\cdots +k_{r}-1)}(\log{M})^{b-1}) \\
         & =
        O(M^{-1}(\log{M})^{b-1}) \to 0
        \qquad (M \to \infty).
    \end{align*}
    This completes the proof.
\end{proof}


\section{Proof of Proposition \ref{prop:Ohno-function-integral}
  and Proposition \ref{prop:connector-convergence}}
\label{sec:app2}

Using Lemma \ref{lem:estimate-below},
we obtain the following estimate of the function $J_{k}(t\,|\,\lambda,\mu)$.

\begin{lem}\label{lem:J-estimate}
    Suppose that
    \begin{align*}
        0<\rho<1/\omega, \quad
        {}-1/\omega<\mathop{\mathrm{Re}}{\lambda}-\rho<0,
        \quad
        {}-1/\omega<\mathop{\mathrm{Re}}{\mu}-\rho<0,
    \end{align*}
    and set $t=-\rho+iu$ with $u\in\mathbb{R}$.
    For $k\ge 1$, set
    \begin{align*}
        A_{k}(\rho\,|\,\lambda, \mu)=
        \frac{2^{(2-k)_{+}}e^{2 \pi \omega(|\mathop{\mathrm{Im}}{\lambda}|+|\mathop{\mathrm{Im}}{\mu}|)}}
        {c(2\pi \omega (\mathop{\mathrm{Re}}{\lambda}-\rho))
        c(2\pi \omega (\mathop{\mathrm{Re}}{\mu}-\rho))
        c(-2\pi \omega \rho)^{(k-2)_{+}}},
    \end{align*}
    where $c(y)$ is defined by \eqref{eq:def-c(y)} and $(m)_{+}=\max{(m, 0)}$.
    Then it holds that
    \begin{align*}
        |J_{1}(t\,|\,\lambda, \mu)|\le
        A_{1}(\rho\,|\,\lambda, \mu)
    \end{align*}
    and
    \begin{align*}
        |J_{k}(t\,|\, \lambda, \mu)|\le
        A_{k}(\rho\,|\,\lambda, \mu)
        e^{-2\pi \omega |u|} \le
        A_{k}(\rho\,|\,\lambda, \mu)
    \end{align*}
    for $k\ge 2$.
\end{lem}

We will use the inequality below.

\begin{lem}\label{lem:estimate-above}
    If $0<\mu<1/4\pi \omega$ and $|z|<\mu$,
    then $|e^{2\pi i \omega z}-1|<6 \pi \omega \mu$.
\end{lem}

\begin{proof}
    Set $x=\mathop{\mathrm{Re}}{z}$ and $y=\mathop{\mathrm{Im}}{z}$.
    For $0 \le t<1/2$, it holds that $1-2t<e^{-t}\le e^{t}<1+2t$.
        {}From the assumption, $|2\pi \omega y|<2\pi \omega \mu<1/2$,
    and hence $1-4\pi \omega \mu<e^{-2\pi \omega y}<1+4\pi \omega \mu$.
    Since $1\ge \cos{(2\pi \omega x)}\ge 1-(2\pi \omega x)^{2}/2\ge 1-2(\pi \omega \mu)^{2}$,
    it holds that
    \begin{align*}
        |e^{2\pi i \omega z}-1|^{2}
         & =
        (e^{-2\pi \omega y}-\cos{(2\pi \omega x)})^{2}+\sin^{2}(2\pi \omega x) \\
         & \le
        (1+4\pi \omega \mu-\cos{(2\pi \omega x)})^{2}+\sin^{2}(2\pi \omega x)  \\
         & =(1+4\pi \omega \mu)^{2}-2(1+4\pi \omega\mu)\cos{(2\pi \omega x)}+1 \\
         & \le
        (1+4\pi \omega \mu)^{2}-2(1+4\pi \omega\mu)(1-2(\pi \omega \mu)^{2})+1 \\
         & =
        4(\pi \omega \mu)^{2}(5+4\pi \omega \mu)<(6\pi \omega \mu)^{2}
    \end{align*}
    since $\pi \omega\mu<1/4$.
\end{proof}

\begin{proof}[Proof of Proposition \ref{prop:Ohno-function-integral}]
    It holds that
    \begin{align}
         &
        \sum_{\substack{0\le m_{1}, \ldots , m_{r}\le M \\
                0\le n_{1}, \ldots , n_{r}\le M}}
        \left| \zeta_{\omega}(k_{1}+m_{1}+n_{1}, \ldots , k_{r}+m_{r}+n_{r})
        \left(\frac{e^{2\pi i \omega \lambda}-1}{2\pi i \omega}\right)^{\sum_{a=1}^{r}m_{a}}
        \left(\frac{e^{2\pi i \omega \mu}-1}{2\pi i \omega}\right)^{\sum_{a=1}^{r}n_{a}}
        \right|
        \label{eq:Ohno-function-pf1}                    \\
         & \le (2\pi \omega)^{|\bk|}
        \left(\prod_{a=1}^{r}\int_{-\epsilon+i\,\mathbb{R}}|dt_{a}|\right)
        \prod_{a=1}^{r}
        \frac{|e^{2\pi i \omega (k_{a}-1)(t_{1}+\cdots +t_{a})}|}
        {|e^{2\pi i \omega t_{a}}-1| |1-e^{2\pi i \omega (t_{1}+\cdots +t_{a})}|^{k_{a}}}
        \nonumber                                       \\
         & \qquad \qquad {}\times
        \sum_{\substack{0\le m_{1}, \ldots , m_{r}\le M \\ 0\le n_{1}, \ldots , n_{r}\le M}}
        \prod_{a=1}^{r}
        \left|\frac{e^{2\pi i \omega \lambda}-1}{e^{-2\pi i \omega(t_{1}+\cdots +t_{a})}-1}
        \right|^{m_{a}}
        \left|\frac{e^{2\pi i \omega \mu}-1}{e^{-2\pi i \omega(t_{1}+\cdots +t_{a})}-1}
        \right|^{n_{a}}.
        \nonumber
    \end{align}
    Since $0<\epsilon<1/2\omega r$, it holds that
    $0<\epsilon/3\pi<1/4\pi\omega$ and
    $0<2\omega\epsilon<1$.
    Hence, from Lemma \ref{lem:estimate-above} and $|\lambda|<\epsilon/3\pi$,
    we see that
    \begin{align*}
        |e^{2\pi i \omega \lambda}-1|<2\omega \epsilon
        <\sin{(\pi \omega \epsilon)}=c(\pi \omega \epsilon).
    \end{align*}
    Since $0<2\pi \omega r \epsilon<\pi$,
    Lemma \ref{lem:estimate-below} implies that
    \begin{align*}
        |e^{-2\pi i \omega (t_{1}+\cdots +t_{a})}-1|\ge
        c(2\pi \omega a \epsilon)\ge c(2\pi \omega \epsilon)
    \end{align*}
    for $t_{1}, \ldots , t_{a} \in -\epsilon+i\,\mathbb{R}$ and $1\le a \le r$.
    Thus we see that
    \begin{align*}
        \left|\frac{e^{2\pi i \lambda}-1}{e^{-2\pi i \omega (t_{1}+\cdots +t_{a})}-1}
        \right|<\frac{c(\pi \omega \epsilon)}{c(2\pi \omega \epsilon)}<1.
    \end{align*}
    This inequality with $\lambda$ replaced by $\mu$ also holds.

    Denote by $E(t_{1}, \ldots , t_{r}|\,\lambda, \mu)$
    the integrand of \eqref{eq:ohno-integral}.
    If $|u|<c<1$, we have
    \begin{align*}
        \sum_{0\le m\le M}|u|^{m}=\frac{1-|u|^{M+1}}{1-|u|} \le
        \frac{1}{1-|u|}\le \frac{1+c}{1-c}\frac{1}{|1-u|}.
    \end{align*}
    By using it, we see that there exists a positive constant
    $C(\omega, r, \epsilon)$ such that
    the integrand of the right hand side of
    \eqref{eq:Ohno-function-pf1} is estimated from above
    on the integral region by
    \begin{align*}
        C(\omega, r, \epsilon)\left|E(t_{1}, \ldots , t_{r}\,|\, \lambda, \mu) \right|,
    \end{align*}
    which does not depend on $M$.
    Set $u_{a}=\mathop{\mathrm{Im}}{u_{a}}$ for $1\le a \le r$.
        {}From Lemma \ref{lem:J-estimate} and $k_{r}\ge 2$, we see that
    \begin{align*}
        \left|E(t_{1}, \ldots , t_{r}\,|\, \lambda, \mu) \right| \le
        \left(\prod_{a=1}^{r}
        \frac{A_{k_{a}}(a\epsilon\,|\,\lambda, \mu)}{c(-2\pi \omega a \epsilon)} \right)
        \exp{(\pi S(u_{1}, \ldots , u_{r}))},
    \end{align*}
    where
    \begin{align*}
        S(u_{1}, \ldots , u_{r})=
        \sum_{a=1}^{r}(u_{a}-|u_{a}|)-2\omega |u_{1}+\cdots +u_{r}|.
    \end{align*}
    It holds that
    \begin{align*}
        S(u_{1}, \ldots , u_{r})
        \le
        \sum_{a=1}^{r}(u_{a}-|u_{a}|)-\omega (u_{1}+\cdots +u_{r})
        \le (|1-\omega|-1)\sum_{a=1}^{r}|u_{a}|.
    \end{align*}
    Hence the function $\exp{(\pi S(u_{1}, \ldots , u_{r}))}$
    is integrable on $\mathbb{R}^{r}$.
    Therefore, by using
    the dominated convergence theorem,
    we see that the infinite sum
    $\mathcal{O}(\bk \,|\, \lambda, \mu)$ is absolutely convergent
    and the equality \eqref{eq:ohno-integral} holds
    if $(\lambda, \mu)$ belongs to the region \eqref{eq:Ohno-region}.

    For $k\ge 1$ and $1\le a \le r$, we have
    \begin{align*}
        \frac{\partial}{\partial \lambda}J_{k}(t_{1}+\cdots +t_{a}\,|\,\lambda, \mu)=
        \frac{2\pi i \omega}{e^{-2\pi i \omega (\lambda+t_{1}+\cdots +t_{a})}-1}
        J_{k}(t_{1}+\cdots +t_{a}\,|\,\lambda, \mu)
    \end{align*}
    and
    \begin{align}
        \left|\frac{1}{e^{-2\pi i \omega (\lambda+t_{1}+\cdots +t_{a})}-1}\right|\le
        \frac{1}{c(2\pi \omega(a\epsilon-\mathop{\mathrm{Re}}{\lambda}))}
        \label{eq:Ohno-function-pf2}
    \end{align}
    in the integral region.
    Therefore, the function $\partial E/\partial \lambda$
    is estimated from above by an integrable function
    which does not depend on $\lambda$ or $\mu$
    in the region \eqref{eq:Ohno-region}.
    We also have a similar estimate for $\partial E/\partial \mu$.
    Therefore the function
    $\mathcal{O}(\bk \,|\, \lambda, \mu)$ is holomorphic
    in the region \eqref{eq:Ohno-region}.
\end{proof}

\begin{proof}[Proof of Proposition \ref{prop:connector-convergence}]
    First we estimate the function $\Theta(t, u, \lambda, \mu)$.
    Set $t=-\epsilon+ix$ and $u=-\rho+iy$ with $\epsilon, \rho>0$ and $x, y \in \mathbb{R}$.
    Suppose that all the real parts of
    $t, u, t+\lambda, t+\mu, u+\lambda, u+\mu$
    belong to the interval $(-2\underline{\omega}, 0)$
    and $-1/\omega<\mathop{\mathrm{Re}}{(t+u+\lambda+\mu)}<0$.
    By using Lemma \ref{lem:estimate-below}, \eqref{eq:G-property} and
    \eqref{eq:G-estimate},
    we see that
    \begin{align}
        \left|\Theta(t, u, \lambda, \mu)\right|\le
        \frac{e^{P(\epsilon, \rho, |\lambda|, |\mu|, \omega)}}
        {c(2\pi \omega(\mathop{\mathrm{Re}}{(\lambda+\mu)-\epsilon-\rho}))}\,
        \exp{(-\pi \omega(\rho (x+|x|)+\epsilon (y+|y|)))},
        \label{eq:theta-estimate}
    \end{align}
    where $P$ is some polynomial and $c(y)$ is defined by \eqref{eq:def-c(y)}.

    We denote by $E(t, u\,|\,\lambda, \mu)$
    the integrand of \eqref{eq:def-connector}.
    Set $x_{a}=\mathop{\mathrm{Im}}{t_{a}} \, (1\le a \le r)$ and
    $y_{b}=\mathop{\mathrm{Im}}{u_{b}} \, (1\le b \le s)$.
    Using \eqref{eq:theta-estimate}, we see that
    \begin{align*}
        |E(t, u\,|\, \lambda, \mu)|
         & \le
        \frac{\prod_{a=1}^{r-1}
            A_{k_{a}}(-a\epsilon\,|\,\lambda, \mu)
            \prod_{b=1}^{s-1}
            A_{l_{b}}(-b\epsilon\,|\,\lambda,\mu)}
        {c(-2\pi \epsilon)^{r+s}
            c(2\pi \omega r\epsilon)^{k_{r}-1}
        c(2\pi \omega s\epsilon)^{l_{s}-1}} \\
         & \quad {}\times
        \frac{e^{P(r\epsilon, s\epsilon, |\lambda|, |\mu|, \omega)}}
        {c(2\pi \omega(\mathop{\mathrm{Re}}{(\lambda+\mu)-(r+s)\epsilon}))}\,
        \exp{(\pi S(x, y))},
    \end{align*}
    where
    \begin{align*}
        S(x, y)=\sum_{a=1}^{r}(x_{a}-|x_{a}|)+\sum_{b=1}^{s}(y_{b}-|y_{b}|)-\omega\epsilon
        \left\{s \left(\sum_{a=1}^{r}x_{a}+\left|\sum_{a=1}^{r}x_{a}\right|\right)+
        r\left(\sum_{b=1}^{s}y_{b}+\left|\sum_{b=1}^{s}y_{b}\right|\right) \right\}.
    \end{align*}
    It holds that
    \begin{align*}
        S(x, y) & \le
        \sum_{a=1}^{r}(x_{a}-|x_{a}|)+\sum_{b=1}^{s}(y_{b}-|y_{b}|)-\omega\epsilon
        \left(\sum_{a=1}^{r}x_{a}+\sum_{b=1}^{s}y_{b}\right)                                          \\
                & \le (|1-\omega\epsilon|-1)\left(\sum_{a=1}^{r}|x_{a}|+\sum_{b=1}^{s}|y_{b}|\right).
    \end{align*}
    Since $0<\epsilon<1/\omega$, we have $|1-\omega\epsilon|-1<0$.
    Hence     $\mathcal{I}(\bk, \bl\,|\,\lambda, \mu)$ is absolutely convergent.
        {}From \eqref{eq:G'/G-estimate}, \eqref{eq:Ohno-function-pf2} and
    the above estimate of the integrand $E(t, u\,|\,\lambda, \mu)$,
    we see that
    the integral $\mathcal{I}(\bk, \bl\,|\,\lambda, \mu)$
    defines a holomorphic function on $B$.
\end{proof}


\end{document}